\documentclass{amsart}

\usepackage{color,graphicx,amssymb,latexsym,amsfonts,txfonts,amsmath,amsthm}
\usepackage{amsmath,amscd}
\usepackage[all,cmtip]{xy}

\usepackage{hyperref}
\hypersetup{
    colorlinks=true,       
    linkcolor=blue,          
    citecolor=blue,        
    filecolor=blue,      
    urlcolor=blue           
}

\newtheorem{theo}{Theorem}[section]
\newtheorem{coro}[theo]{Corollary}
\newtheorem{prop}[theo]{Proposition}
\newtheorem{lemma}[theo]{Lemma}

\theoremstyle{remark}
\newtheorem{remark}{\bf Remark}

\def\h{{\mathbb H}}

\begin{document}

\title{Generalized quasi-dihedral group as automorphism group of Riemann surfaces}

\author{Rub\'en A. Hidalgo, Yerika Mar\'in Montilla and Sa\'ul Quispe}


\address{Departamento de Matem\'atica y Estad\'{\i}stica, Universidad de La Frontera. Temuco, Chile}
\email{ruben.hidalgo@ufrontera.cl}
\email{yerika.marin@ufrontera.cl}
\email{saul.quispe@ufrontera.cl}
\thanks{Partially supported by Projects FONDECYT Regular N. 1220261 and 1190001. The second author has been supported by ANID/Beca de Doctorado Nacional/ 21190335.}

\begin{abstract}
In this paper, we discuss certain types of conformal/anticonformal actions of the generalized quasi-dihedral group $G_{n}$ of order $8n$, for $n\geq 2$, on closed Riemann surfaces, pseudo-real Riemann surfaces and compact Klein surfaces, and in each of these actions we study  the uniqueness (up to homeomorphisms) action problem.
\end{abstract}

\maketitle

\section{Introduction}
Let $S$ be a Riemann surface and ${\rm Aut}^{+}(S)$ (respectively, ${\rm Aut}(S)$) be its group of conformal (respectively, conformal and anticonformal) automorphisms. In the generic situation, ${\rm Aut}(S)={\rm Aut}^{+}(S)$; otherwise, ${\rm Aut}^{+}(S)$ is a subgroup of index two. In \cite{Schwarz}, Schwarz proved that, if $S$ is a closed Riemann surface of genus $g \geq 2$, then ${\rm Aut}^{+}(S)$ is finite and later Hurwitz \cite{Hurwitz} obtained the upper bound  $|{\rm Aut}^{+}(S)| \leq 84(g-1)$.  
Riemann surfaces with non-trivial group of automorphisms define the branch locus ${\mathcal B}_{g}$ of the 
moduli space ${\mathfrak M}_{g}$, of biholomorphism classes of closed Riemann surfaces of genus $g \geq 2$, known to be a complex orbifold of dimension $3(g-1)$ \cite{Nag}. If $g \geq 4$, then ${\mathcal B}_{g}$ coincides with the locus where ${\mathfrak M}_{g}$ fails to be a topological manifold. The space 
${\mathfrak M}_{g}$ has a natural real structure (this coming from complex conjugation). The fixed points of such a real structure, its real points, are the (isomorphism classes) closed Riemann surfaces admitting anticonformal automorphisms. A closed Riemann surface that admits a reflection (i.e., an anticonformal involution with fixed points) as an automorphism is called {\em real Riemann surface}; otherwise, it is called {\em pseudo-real Riemann surface}. Pseudo-real Riemann surfaces are examples of Riemann surfaces which cannot be defined over their field of moduli \cite{AQR}. 
In general, a finite group might not be realized as the group of conformal/anticonformal automorphisms, admitting anticonformal ones, of a {pseudo-real Riemann surface} (in \cite{Bu10}, it was observed that a necessary condition for that to happen is for the group to have order a multiple of $4$).

Let us fix a finite abstract group $G$. 
In \cite{Greenberg}, Greenberg proved that there are closed Riemann surfaces $S$ such that 
 $G$ can be see as a subgroup of ${\rm Aut}^{+}(S)$ (conformal action). 
If, moreover, $G$ contains an index two subgroup $H$, then it is possible to find $S$ such that $G<{\rm Aut}(S)$ and $H=G \cap {\rm Aut}^{+}(S)$
(conformal/anticonformal action).
These facts permit to define the {\it strong symmetric genus} $\sigma^{0}(G)$ (respectively, {\em symmetric genus} $\sigma(G)$) of $G$ 
as the minimal genus for a conformal (respectively, conformal/anticonformal) action of $G$ \cite{Burn,Hurwitz,T}.
We note that $\sigma(G) \leq \sigma^{0}(G)$ (if $G$ has no index two subgroups, then the equality holds).
Groups for which $\sigma^{0}(G)=0$ are given by  the cyclic groups, the dihedral groups and the Platonic solid symmetry groups. Those with $\sigma^{0}(G)=1$ are also known \cite{GroT}. As a consequence of the Hurwitz's bound, there are only finitely many (up to isomorphisms) groups of a given strong symmetric genus at least two. It is also known that every integer at least two  is the strong symmetric genus for some finite group \cite{May3}. 
A conformal action of $G$, over a closed Riemann surface, is called {\it purely-non-free} if every element acts with a non-empty set of fixed points.
In \cite{BGH}, it was observed that $G$ acts purely-non-free on some Riemann surface. This permits to define 
the {\it pure symmetric genus}  $\sigma_{p}(G)$ as the minimal genus on which $G$ acts purely-non-free  ($\sigma^{0}(G)\leq \sigma_{p}(G)$).
A compact Klein surface $X$ is canonically doubly covered by a Riemann surface of some genus $g\geq 2$ (also called the {\em algebraic genus} of $X$). Its underlying topological surface is either (i) non-orientable with empty boundary ({\em closed Klein surface}) or (ii) it has non-empty boundary (and can be orientable or not). 
If $k$ is the number of boundary components of $X$, then its topological genus is $\gamma=(g-k+1)/\eta$, where $\eta= 2$ if $X$ is an orientable surface and $\eta=1$ otherwise.
It is known that 
$G$ acts as a group of automorphisms of bordered Klein surfaces (i.e., $k>0$). The minimum algebraic genus $\rho(G)$ of these surfaces is called the {\em real genus} of $G$ \cite{May93}. Also, it is known that $G$ acts as a group of automorphisms of closed Klein surfaces (i.e., $k=0$). The minimum topological genus $\tilde{\sigma}(G)$ of these surfaces is called the {\em symmetric crosscap number} of $G$ \cite{May0}.

In this paper, we study conformal/anticonformal actions of the generalized quasi-dihedral group $G_{n}$, of order $8n$, on closed Riemann surfaces, for $n\geq 2$. Below, we list some of the main results of this paper.
In Corollary \ref{strong}, we observe that $\sigma^{0}(G_{n})$ is equal to $n$. Such a minimal conformal action happens for the quotient orbifold $S/G_{n}$ of signature $(0;2,4,4n)$ (this was previously observed in \cite{MZ7} for $n$ a power of two and in \cite{May3} for $n$ odd). In Theorem \ref{pta}, we prove that, up to homeomorphisms, there is only one triangular action of $G_{n}$ for every $n$. These triangular action is produced in a familiar kind of hyperelliptic Riemann surfaces (Wiman curves of type II \cite{Wi}). They are described, in terms of the corresponding monodromy group and the associated bipartite graph,  in 
Remark \ref{re1c}. In Theorem \ref{pta}, we observe that the triangular action of $G_{n}$ on a closed Riemann surface of genus $n$ is purely-non-free for $n$ even, and it is not purely-non-free for $n$ odd. On the other hand, in Corollary \ref{strong}, we observe that $\sigma_{p}(G_{n})$ is equal to $n$ for $n$ even, and $3n$ for $n$ odd.
In Theorem \ref{VJ}, we describe the isotypical decomposition of the Jacobian variety $J_S$, induced by the triangular action of $G_{n}$. 
As $G_{n}$ has index two subgroups (Lemma \ref{pgg1}), it can be realized as the full group of conformal/anticonformal automorphisms of suitable Riemann surfaces, such that it admits anticonformal ones. As consequence of Proulx class \cite{GroT}, it is well known that $G_{n}$ acts with such a property in genus $\sigma(G_{n})=1$.  In Theorem \ref{hyp}, we observe that the next minimal genus $\sigma^{hyp}(G_{n})\geq 2$ over which $G_{n}$ acts as a group of conformal/anticonformal automorphisms, and admitting anticonformal ones, is $n$ for $n$ even, and $n-1$ for $n$ odd. This permits to observe that for any non-negative integer $g\geq 2$, there is at least one group of symmetric hyperbolic genus $g$ (see, Corollary \ref{shg}).
In Theorem \ref{tumshg}, we prove that this minimal action is unique, up to homeomorphisms. 
We also given integers $k \geq 3$ and $n\geq 2$, in Theorem \ref{tps}, we construct pseudo-real Riemann surfaces, of genus $g=2nk-4n+1$, whose full group of conformal/anticonformal automorphisms is $G_{n}$ (by results in \cite{Hu1}, these Riemann surfaces are non-hyperelliptic). The minimal genus of pseudo-real Riemann surfaces with $G_{n}$ as its full group of automorphisms, is $2n+1$, this was previously proved in \cite[Proposition 3.2]{CL} (see, Proposition \ref{pmprs}). In Corollary \ref{Ups}, we prove that this minimal action is unique up homeomorphisms. 
For integers $n\geq 3$ odd, $l\geq 2$ and $r\geq 1$ odd, we observe that there are pseudo-real Riemann surfaces of genus $4nl+6nr-8n+1$ with $G_{n}$ as group of conformal automorphisms (Theorem \ref{tps1}). Moreover, in this case, in Theorem \ref{mps+}, we prove that the minimal genus of pseudo-real Riemann surfaces, with $G_{n}$ as group of conformal automorphisms,  is $6n+1$ (this minimal action is not unique, see Remark \ref{Ups+}).  For $n$ even, in Theorem \ref{tps2}, we observe that the group $G_{n}$ cannot be realized as group of conformal automorphisms of pseudo-real Riemann surfaces. However, in Theorem \ref{ejemplo}, we constructed pseudo-real Riemann surfaces $S$ admitting $G_{n}$ as an index two subgroup of ${\rm Aut}^{+}(S)$. 
In Theorem \ref{real}, we observe that (a) $\rho(G_{n})$ is equal to $2n+1$ and the action is unique
and (b) $\tilde{\sigma}(G_{n})$ is equal to $2n+2$ but the action is not unique for $n \neq 3$. 
Part (a) was previously observed in \cite{May94} for $n$ a power of two and in \cite{EM} for $n$ odd.

\medskip
\noindent{\bf Notation.} Throughout this paper we denote by $C_n$ the cyclic group of order $n$, by $D_{2m}$ the dihedral group of order $2m$, by $DC_n$ the dicyclic group of order $n$.

\section{Preliminaries}
\subsection{Generalized quasi-dihedral group}\label{Section 3}
The {\it generalized quasi-dihedral group} of order $8n$ is 
 \begin{equation}\label{g1}
G_{n}:=\langle x, y:\  x^{4n}=y^{2}=1, \,  yxy=x^{2n-1}\rangle=C_{4n}\rtimes_{2n-1}C_{2}.
\end{equation} 
 
Using the relation $yx^{k}=x^{(2n-1)k}y$, for $k \geqslant 1$, one may check the following properties.
 
\begin{lemma}\label{Proposition2}\label{pgg1} 
Let $n\geqslant 2$ be an integer. Then 
\setlength{\leftmargini}{8mm}
\begin{enumerate}
\item The group $G_{n}$ is a non-abelian group of order $8n$, every element has a unique presentation of the form $y^{j}x^{i}$, where $j$ and $i$ are integers with $ j \in\lbrace 0, 1\rbrace$ and $0\leq i \leq 4n$, and these have order
$$\mid x^{i}\mid =\dfrac{4n}{{\rm gcd}(i, 4n)},\quad  \mid yx^{i}\mid=\begin{cases} 2, & \text{if\ $i$ either even or zero},\\ 4, & \text{if\ $i$ odd}.\end{cases}$$

 \item The index two subgroups of $G_{n}$ are exactly the following ones
 \begin{center}$ C_{4n}=\langle x\rangle,\qquad \ D_{4n}=\langle x^{2}, y \rangle, \qquad \ DC_{4n}=\langle x^{2}, yx \rangle.$\end{center}
 \item The number of involutions of the groups $G_{n}$ and $D_{4n},$ is $2n+1$ and, the number of involutions of the groups $C_{4n}$ and $DC_{4n}$ is one.

\item  For $n$ even there are exactly $2n+3$ conjugacy classes of $G_{n}$, with representatives given in the following table 
\begin{center}
\begin{tabular}{|c|c|c|c|c|c|c|c|c|}
\hline 
Rep. & $1$ & $x$ & $x^{2}$ & $\cdots$ & $x^{2n-1}$ & $x^{2n}$ & $y$ & $yx$ \\ 
\hline 
size & 1 & 2 & 2 & $\cdots$ & 2 & 1 & $2n$ & $2n$ \\ 
\hline 
\end{tabular} 
\end{center}
\item   For $n$ odd there are exactly $2n+6$ conjugacy classes of $G_{n}$, with representatives given in the following table
\begin{center}

\scalebox{0.75}{\begin{tabular}{|c|c|c|c|c|c|c|c|c|c|c|c|c|c|c|c|c|c|c|c|}
\hline 
Rep. & $1$ & $x$  & $\cdots$ & $x^{n-1}$ &$x^{n}$ & $x^{n+1}$ & $x^{n+3}$ & $\cdots$ & $x^{2n-1}$&$x^{2n}$ &$x^{2n+1}$ & $\cdots$ & $x^{3n-1}$ &$x^{3n}$ & $y$ & $yx$& $yx^{2}$ &  $yx^{3}$ \\ 
\hline 
size & 1 & 2 &  $\cdots$ & 2 &1  & 2 &2& $\cdots$ & 2& 1 & 2& $\cdots$ & 2 & 1& $n$ & $n$ & $n$ &  $n$ \\ 
\hline 
\end{tabular}} 
\end{center}

\item The automorphisms of the group $G_{n}$ are given by 
\begin{center}$\psi_{u, 2v}(x)=x^{u} \quad \text{and}\quad  \psi_{u, 2v}(y)=yx^{2v},$\end{center} 
where $u \in \{1,\ldots,4n-1\}$ is such that ${\rm gcd}(u, 4n)=1$ and $v\in \{0,1,\ldots,2n-1\}$. The group ${\rm Aut}(G_{n})$ has order $\phi(4n)\cdot 2n$, where $\phi$ is the Euler function.
\end{enumerate}
\end{lemma} 

\subsection{NEC groups}\label{S2}
Let $\mathcal{L}$ be the group of isometries of the hyperbolic upper half-plane $\h$ and $\mathcal{L}^{+}$ be its index two subgroup of orientation-preserving elements.
An NEC {\em group} is a discrete subgroup $\Delta$ of $\mathcal{L}$ such that the quotient space $\h/\Delta$ is a compact surface. An NEC group contained in $\mathcal{L}^{+}$ is called a {\em Fuchsian group}, and a {\em proper} NEC {\em group} otherwise. If $\Delta$ is a proper NEC group, then $\Delta^{+}=\Delta\cap \mathcal{L}^{+}$ is called its {\em canonical Fuchsian subgroup}. Note that $[\Delta:\Delta^{+}]=2$ and $\Delta^{+}$ is the unique subgroup of index 2 in $\Delta$ contained in $\mathcal{L}^{+}$.
In general, the algebraic structure of an NEC group $\Delta$ is described by the so-called {\em signature} $s(\Delta)$ \cite{Mac, Wil}:
\begin{equation}\label{E21} s(\Delta)=(h;\pm;[m_1,\cdots, m_r];\{(n_{11},\cdots, n_{1s_1}),\cdots, (n_{k1},\cdots, n_{ks_k})\}),\end{equation} 
where $h,r,k,m_i, n_{ij}$ are integers with $h, r, k\geq 0$ and $m_i, n_{ij}>1$ for all $i, j$. Here $h$ is the topological genus of the surface $\h/\Delta$ and $``+"$ means that $\h/\Delta$ is {\em orientable}, and $``-"$ means that $\h/\Delta$ is {\em non-orientable}. The number $k$ is the number of connected boundary components of $\h/\Delta$. We call $m_i$ the {\em proper periods}, $n_{ij}$ the {\em periods}, and $(n_{i1},\cdots, n_{is_i})$ the {\em period-cycles} of $s(\Delta)$. We will denote by $[-]$, $(-)$ and $\{-\}$ the cases when $r=0$, $s_i=0$ and $k=0$, respectively. When there are no proper periods and there are no period-cycles in $s(\Delta)$ we say $\Delta$ is a {\em surface group}.
The signature provides a presentation of $\Delta$ \cite{Mac, Wil}, by {\em generators}:
\begin{eqnarray*}
\text{(elliptic generators)} && \beta_i\in \Delta^{+} \quad (i=1,\cdots, r); \\
\text{(reflections)} && c_{ij}\in \Delta\setminus \Delta^{+}\quad (i=1,\cdots, k;\ j=0,\cdots, s_i); \\
\text{(boundary generators)} && e_i\in \Delta^{+}\quad (i=1,\cdots, k); \\
 \text{(hyperbolic generators)} && a_i,b_i\in \Delta^{+}\quad (i=1,\cdots, h),\ \text{if $\h/\Delta$ is orientable}; \\
  \text{(glide reflections generators)} && d_i\in \Delta\setminus\Delta^{+}\quad (i=1,\cdots, h), \ \text{if $\h/\Delta$ is non-orientable};
  \end{eqnarray*}
   and {\em relations}:
   \begin{eqnarray*}
&& \beta_i^{m_i}=1\quad (i=1,\cdots, r);\\
&& e_ic_{i0}e_i^{-1}c_{is_i}=1\quad (i=1,\cdots, k);\\
&& c_{ij-1}^{2}=c_{ij}^2=(c_{ij-1}c_{ij})^{n_{ij}}=1\quad (i=1,\cdots, k;\ j=1,\cdots, s_i);\\
\text{(long relation)}&& \prod_{i=1}^r\beta_i\prod_{i=1}^ke_i\prod_{i=1}^h[a_i,b_i]=1,\quad \text{if $\h/\Delta$ is orientable};\\
\text{(long relation)} && \prod_{i=1}^r \beta_i\prod_{i=1}^ke_i\prod_{i=1}^hd_i^2=1, \quad \text{if $\h/\Delta$ is non-orientable},
\end{eqnarray*} where $1$ denotes the identity map $id_{\h}$ in $\h$ and $[a_i,b_i]=a_ib_ia_i^{-1}b_i^{-1}$. 

The hyperbolic {\em area} of $\Delta$ with signature (\ref{E21})
is the {\em hyperbolic area} of any fundamental region for $\Delta$, and is given by \begin{equation}\label{E22} \mu (\Delta)=2\pi\Big(\eta h+k-2+\sum_{i=1}^r\big(1-\frac{1}{m_i}\big)+\frac{1}{2}\sum_{i=1}^k\sum_{j=1}^{s_i}\big(1-\frac{1}{n_{ij}}\big)\Big),\end{equation} with $\eta=2$ or 1 depending on whether or not $\h/\Delta$ is orientable. An NEC group with signature (\ref{E21}) actually exists if and only if the right-hand side of (\ref{E22}) is greater than 0.
The {\em reduced area} of an NEC group $\Delta$, denoted by $|\Delta|^*$, is given by $\mu(\Delta)/2\pi$. If $\Gamma$ is a subgroup of $\Delta$ of finite index, then the Riemann-Hurwitz formula holds $[\Delta:\Gamma]=\frac{\mu(\Gamma)}{\mu(\Delta)}.$

\subsection{Topologically equivalent conformal/anticonformal actions}\label{Sec:signatures}
Let $S$ be a closed Riemann surface $S$ of genus $g \geq 2$. By the uniformization theorem, up to biholomorphisms, $S=\h/K$, where $K$ is a Fuchsian surface group.
We say that a finite group $G$  {\em acts} as a group of conformal (respectively, conformal/anticonformal) automorphisms of $S$ if it can be realizable as a subgroup of ${\rm Aut}^{+}(S)$ (respectively, ${\rm Aut}(S)$).
This is equivalent to the existence of a Fuchsian (respectively, an NEC) group  $\Delta$, containing $K$ as a normal subgroup, and of an epimorphism $\theta: \Delta \to G$ whose kernel is $K$, we say that $\theta$ provides a conformal (respectively, conformal/anticonformal) action of $G$ in $S$. 
Two conformal/anticonformal actions $\theta_1$, $\theta_2$ are {\em topologically equivalent} if there is an $\omega\in {\rm Aut}(G)$ and an $h\in {\rm Hom}^{+}(S)$ such that $\theta_2(g)=h\theta_1(\omega(g))h^{-1}$ for all $g\in G$.
This is equivalent to the existence of automorphisms $\phi \in {\rm Aut}(\Delta)$ and $\omega \in {\rm Aut}(G)$ such that
$\theta_2=\omega\circ \theta_1\circ \phi^{-1}$.

If $G$ acts conformally, then $S/G=\h/\Delta$ is a closed Riemann surface of genus $h$ and it has exactly $n$ cone points of respective cone orders $k_{1},\ldots, k_{n}$; the tuple $(h;k_{1},\ldots,k_{n})$ is also called the {\em signature} of $S/G$. If $h=0$ and $n=3$, then we talk of a {\it triangular} action.

\subsection{Dessins d'enfants}\label{Sec:dessins}
A {\it dessin d'enfant} corresponds to a bipartite map ${\mathcal G}$ on a closed orientable surface $X$. These objects were studied as early as the nineteenth century and rediscovered by Grothendieck in the twentieth century in his ambitious research outline ~\cite{Gro} (see also, the recent books \cite{GiGo,JW}). The dessin d'enfant
defines (up to isomorphisms) a unique Riemann surface structure $S$ on $X$, together with a holomorphic branched cover $\beta:S \to \widehat{\mathbb C}$  whose branch values are containing in the set $\{\infty,0,1\}$ ($\beta$ is called a {\em Belyi map}, $S$ is a {\em Belyi curve} and $(S,\beta)$ is a {\em Belyi pair}). Conversely, to every Belyi pair $(S,\beta)$ there is associated a dessin  d'enfant in $S$ ($\beta^{-1}(0)$ and $\beta^{-1}(1)$ provide, respectively, the white and black vertices, and the edges being $\beta^{-1}([0,1])$). 
The dessin is called {\it regular} if its group of automorphisms $G$ (i.e, the deck group of $\beta$) acts transitively on the set of edges (equivalently, $\beta$ is a Galois branched covering). In this case, $G$ defines a triangular action and, in particular, $G$ is generated by two elements (in fact, every finite group generated by two elements appears as the group of automorphisms of some regular dessin). Examples of such type of groups are the generalized quasi-dihedral groups $G_{n}$ (see, Section \ref{Section 3}).
There is a bijection between (i) equivalence classes of regular dessins d'enfants with group of automorphisms $G$ and (ii) $G$-conjugacy classes of pairs of generators of $G$.
 
\section{Triangular conformal actions of $G_{n}$}
In this section, we describe the triangular conformal actions of $G_{n}$. First, we observe the well know family $S_n$ of hyperelliptic Riemann surface admitting a triangular conformal action of $G_{n}$. 

\subsection{The strong and pure symmetric genus of $G_{n}$}
If $n\geq 2$, then the Riemann surface $S_{n}$ of genus $n$, defined by the algebraic curve $w^{2}=z(z^{2n} -1)$, is called the Wiman curve of type II \cite{Wi}. 
Some conformal automorphisms of $S_n$ are given by $x(z, w) = (\rho_{2n} z, \rho_{4n} w)$, $y (z, w) = (\rho_{2n}/ z, i\rho_{4n}w / z^{n+1})$ with $(\rho_{2n})^{n}=-1$ and $\langle x,y \rangle \cong G_{n}$. The quotient orbifold $S_n/\langle x, y \rangle$ has signature $(0; 2,4, 4n)$. If $n\geq 3$, then this is a maximal signature  \cite{S.}, so ${\rm Aut}^{+}(S_n)=\langle x, y \rangle$. If $n=2$, then $S_n$ has as an extra conformal automorphism of order three, given by $t(z, w) = (i (1-z) / (1 + z), 2 (1 + i) w / (z + 1)^{3})$. In this case, ${\rm Aut}^{+}(S_n) =\langle x, y, t \rangle$ (a group of order 48). 

The bipartite graph, associated to the regular dessin d'enfant on $S_n$ induced by $G_{n}$, is the graph $K_{2, 2n}^2$, which is obtained from the complete bipartite graph (see, \cite[pp. 17]{D}) $K_{2, 2n}$ in which each of its edges is replaced by two edges. In order to observe this, consider a regular branch cover $\beta:S_n\to \widehat{\mathbb{C}}$, whose deck group is $G_{n}$. We may assume that $\beta=Q\circ P$, where $P:S_n\to \widehat{\mathbb{C}}$ has deck group $C_{4n}=\langle x\rangle$, and $Q(x)=x^2$, whose deck group is $G_{n}/\langle x\rangle\cong C_2$ (see also, Remark \ref{re1c} in where the corresponding monodromy group is described).

\begin{theo}\label{pta}
Let $S$ be a closed Riemann surface such that ${\rm Aut}^{+}(S)=G_{n}$, for $n \geq 2$, and such that $S/G_{n}$ has triangular signature.
Then  
\setlength{\leftmargini}{8mm}
\begin{enumerate} 
\item[ (a)] $S/G_{n}$  has signature $(0; 2,4,4n)$ and $S$ is isomorphic to $S_n$.
\item[ (b)] If $n\geq 2$ even, the action of $G_{n}$ is purely-non-free.
\item[ (c)] If $n\geq 3$ odd, the action of $G_{n}$ is not purely-non-free.
\end{enumerate}
\end{theo}
  
\begin{proof}
Let us consider the presentation \eqref{g1} of generalized quasi-dihedral group $G_{n}$. Assume that $G_{n}$ acts in a triangular way on the closed Riemann surface $S$.

\medskip
\noindent ({\bf a})  Let $\pi: S \rightarrow \mathcal{O}= S/\langle x \rangle$ be a branched regular cover map with deck group $\langle x \rangle$. As $\langle x \rangle$ is a normal subgroup of $G_{n}$ and $y^{2}=1\in \langle x \rangle$, the automorphism $y$ induces a conformal involution $\widehat{y}$ of the quotient orbifold $\mathcal{O}$, so it permutes the branch values of $\pi$ (i.e., the cone points of $\mathcal{O}$) and $S/G_{n} = \mathcal{O}/ \langle \widehat{y} \rangle$. The triangular property of the action of $G_{n}$ on $S$ (together with the Riemann-Hurwitz formula) ensures that $\mathcal{O}$ has genus zero and that its set of cone points are given by a pair of points $p_{1}, p_{2}$ (which are permuted, but not fixed, by $\widehat{y}$) and one or both of the fixed points of $\widehat{y}$.
By the uniformization theorem, we may identify $\mathcal{O}$ with the Riemann sphere $\widehat{\mathbb{C}}$. Up to post-composition of $\pi$ with a suitable M\"obius transformation, we may also assume $\widehat{y}(z) = -z,$ $p_{1}= -1$ and  $p_{2}= 1$. As the finite groups of M\"obius transformations are either cyclic, dihedral, ${\rm A_{4}}$, ${\rm A_{5}}$ or ${\rm S_{4}}$, and as $G_{n}$ is not isomorphic to any of them, the surface $S$ cannot be of genus zero. This (together with the Riemann-Hurwitz formula) asserts that $\pm 1$ are not the only cone points of $\mathcal{O}$, at least one of the two fixed points of $\widehat{y}$ must also be a cone point (we may assume that $0$ is another cone point of $\mathcal{O}$). For the point $\infty$  we have the following.

\medskip
\noindent{\bf Case $1$} ($\infty$ is a cone point of $\mathcal{O}$). In this case, the cone points of $\mathcal{O}$ are $\pm 1$, $0$ and $\infty$. As $0$ and $\infty$ are fixed points of $\widehat{y}$, each one them has points in its preimage on $S$ with $G_{n}$-stabilizer generated by an element of the form $yx^{k}$ with $k$ odd (if $k$ is either even or zero, then a point in the preimage of $0$ or $\infty$ induces a conic point of order 2 on $S/G_{n}$ which contradicts the fact that $0$ and $\infty$ are fixed points of $\widehat{y}$ and that them are cone points of $\mathcal{O}$). Thus, the points in the preimage of 0 and $\infty$ induce conic points of order $4$ on $S/G_{n}$, and that $0$ and $\infty$ are conic points of order $2$ in $\mathcal{O}$. 
  
 As $x^{2n}\in \langle yx^k\rangle$ with $k$ odd, and this does not generate $\langle x\rangle$, the $\langle x\rangle$-stabilizer of any point on $-1$ (and also about $1= \widehat{y}(-1)$) must be stabilized by a non-trivial power $x^{m}$ (where we assume $m$ is a divisor of $4n$) such that $\langle x^{m}, x^{2n} \rangle =\langle x \rangle$. From where $m$ must be prime relative to $2n$, so we have $m=1$. Thereby, $\mathcal{O}$ has signature $(0;2,2, 4n, 4n)$ and $S/G_{n}$ has signature $(0;4, 4, 4n).$ But the signature $(0;4, 4, 4n)$ is not admissible for the action of $G_{n}$ on $S$ (i.e., there is no an epimorphism from a Fuchsian group with signature $(0;+;[4,4,4n];\{-\})$ on  $G_{n}$).

\medskip
\noindent{\bf Case $2$} ($\infty$ is not a cone point of $\mathcal{O}$). In this case the $G_{n}$-stabilizer of any point on $\infty$ does not contain a non-trivial power of $x$, this has points in its preimage with $G_{n}$-stabilizer generated by an element of the form $y$ or  $yx^{k}$ with $k$ even. So, $\infty$ induces a conic point of order 2 on $S/G_{n}.$ Thereby, the cone points of $\mathcal{O}$ are $\pm 1$ and $0$. As $0$ is fixed point of  $\widehat{y}$, this has points in its preimage on $S$ with $G_{n}$-stabilizer generated by an element of the form $yx^{k}$ with $k$ odd (if $k$ is either even or zero, then a point in the fiber of $0$ induces a conic point of order 2 on $S/G_{n}$ which contradicts the fact that $0$ is a fixed point of $\widehat{y}$ and that this is a cone point of $\mathcal{O}$). Thus, the points in the preimage of $0$ produce a conic point of order $4$ on $S/G_{n}$ and that $0$ is a conic point of order $2$ in $\mathcal{O}$. 

As $x^{2n}\in \langle yx^k\rangle$ with $k$ odd, and this does not generate $\langle x\rangle$. A similar argument as in the Case 1 is used to obtain that $\mathcal{O}$ has signature $(0;2, 4n, 4n)$ and $S/G_{n}$ has signature $(0;2, 4, 4n).$ The signature $(0;2, 4, 4n)$ is admissible for the action of $G_{n}$ on $S$ (i.e., there is an epimorphism $\theta$ of a Fuchsian group with signature $(0;+;[2,4,4n];\{-\})$ on $G_{n}$ with torsion-free kernel $K={\rm ker}(\theta)$ such that $S=\h/K$). 
  
 \medskip  
 \noindent ({\bf b}) Let $n$ even and $\pi:S\to \widehat{\mathbb{C}}$ be a branched regular cover map with deck group $\langle x\rangle$ as in (a). Then can be seen that the number of fixed points of $x$ is exactly two and $x^{2n}$ has exactly $2+2n$ fixed points. As $(yx)^{2}=x^{2n}$, the automorphism $yx$ induce an involution $\widehat{yx}$ of the quotient orbifold $\mathcal{O}$ with exactly two fixed points. On the other hand, since the point $0$ is the projection of the fixed points of $(yx)^2$ (proof of the item (a)), and the fact that the number of elements in the conjugacy class of $yx$ is $2n$, we have that $yx$ have exactly two fixed points. As $y$ and $yx^k$ with $k$ even, are in the same conjugacy class (Lemma \ref{pgg1}), from the proof of the item (a) (i.e., the point $\infty$ have points in its preimage with $G_{n}$-stabilizer generated by elements of the form $y$ or $yx^k$ with $k$ even), we conclude that every element of $G_{n}$ acts with fixed points.

\medskip  
\noindent ({\bf c}) Let $n$ odd and $\pi:S\to \widehat{\mathbb{C}}$ be a branched regular cover map with deck group $\langle x\rangle$ as in (a). Then can be seen that the number of fixed points of $x$ is exactly two and that $x^{2n}$ has exactly $2+2n$ fixed points. As $(yx)^{2}=x^{2n}$, the automorphism $yx$ induce an involution $\widehat{yx}$ of the quotient orbifold $\mathcal{O}$ with exactly two fixed points. On the other hand, since the point $0$ is being the projection of the fixed points of $x^{2n}$ (proof of the item (a)), and the fact that the number of elements in the conjugacy class of $yx$ (and of $yx^3$) is $n$, we have that $yx$ (also $yx^3$) must have exactly two fixed points. As $y$ and $yx^2$ not are in the same conjugacy class (Lemma \ref{pgg1}), from the proof of the item (a), we have that $y$ or $yx^2$ acts with fixed points. 
 \end{proof}
  
 \begin{remark}
Another (singular) model of $S$ is given by the affine algebraic curve $v^{4n}=u^{2n}(u-1)(u+1)^{2n-1}$ with conformal automorphisms $x(u, v)=(u, \rho_{4n}v)$ and $y(u, v)=(-u, v^{2n-1}/u^{n-1}(u+1)^{n-1})$ such that $\langle x, y\rangle\cong G_{n}$. In fact, as $S/G_{n}$ has signature $(0;2, 4, 4n)$ (by the Riemann-Hurwitz formula), $S$ has genus $n$. As seen in the proof of part (a) of Theorem \ref{pta},  $\pi: S \rightarrow \widehat{\mathbb{C}}$ is a cyclic branched regular covering, branched at the points $\pm 1$ (with branching order $4n$) and at the point $0$ (with branching order $2$). In particular, an equation for $S$ must be of the form $w^{4n}=t^{\alpha}(t-1)^{\beta}(t+1)^{\gamma},$ where $\alpha, \beta, \gamma \in \lbrace 1, \ldots ,4n-1 \rbrace$ are such that: (i) $\beta$ and $\gamma$ are both relative primes to $4n$, (ii) ${\rm gcd}(\alpha,4n)=2n$, (iii) $\alpha+\beta + \gamma \equiv 0 \, {\rm mod} \, (4n)$. In this model, $\pi$ corresponds to the projection $(t, w) \mapsto t$. Condition (ii) implies that $\alpha=2n$. By the condition (iii), we may suppose, without loss of generality, that $\beta=1$; so $\gamma= 2n-1$. In this way, we have obtained the uniqueness, except for isomorphisms of $S$.
 \end{remark}
      
\begin{remark}\label{re1c}
Let us consider the following permutations of $\mathfrak{S}_{8n}$
$$\eta=(1,2,\ldots,4n)(4n+1, 4n+2,\ldots, 8n),$$
$$\sigma=\prod_{1=k=2l+j}^{2n}(k,8n-((1-j)2n+(k-1))\prod_{2n+1=k=2l+j}^{4n}(k,8n-((j-1)2n+(k-1)).$$
Then, $\eta^{4n}=\sigma^{2}=1$,  $\sigma \eta \sigma=\eta^{2n-1}$, and  $\langle \eta, \sigma \rangle \cong G_{n},$ where the isomorphism is the one taking $\eta$ to $x$ and $\sigma$ to $y$.  If $\tau=\sigma \eta^{4n-1}$, then $\sigma \tau \eta =1$, and the pair $(\eta, \tau)$ determines the monodromy group associated to the regular dessin d'enfant of signature $(0;2, 4, 4n)$ as described in the Theorem \ref{pta}. This permits to see that the associated bipartite graph of this dessin d'enfant is $K_{2,2n}^2$. For $n=3$, we have the following dessin d'enfant (see, Figure 1).
 \end{remark}  
 
 \begin{center}
	\includegraphics[width=6cm]{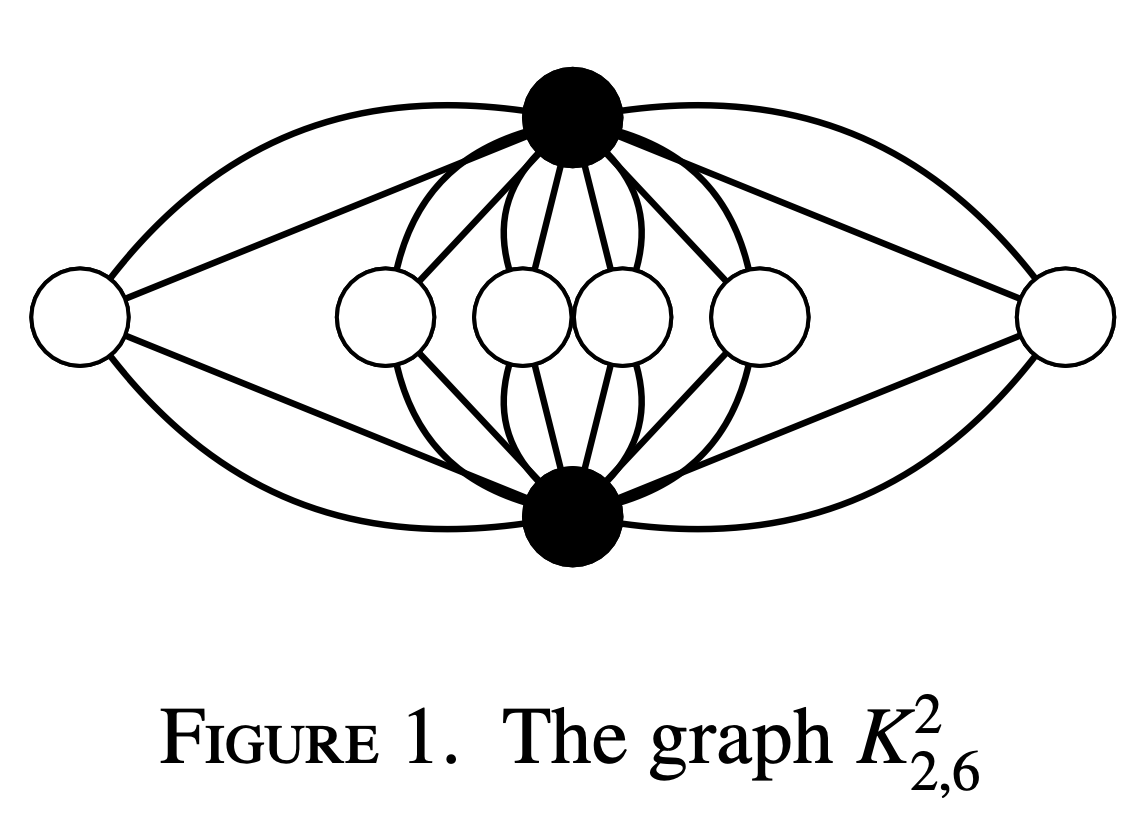}
	
	\end{center} 
 
\begin{coro}[Strong and pure symmetric genus of $G_{n}$]\label{strong} 
If $n\geq 2$, then 
\begin{enumerate}
\item $\sigma^{0}(G_{n})=n$ and, up homeomorphisms, the conformal action of $G_{n}$ is unique.     
\item $\sigma_p(G_{n})= \left\{\begin{array}{ll}
n, & \text{if $n\geq 2$ even}\\
3n, & \text{if $n\geq 3$ odd}
\end{array}
\right\}$ 
and, up homeomorphisms, the conformal action of $G_{n}$ is unique in each case.  
\end{enumerate}
\end{coro}
\begin{proof}
(1) 
By Theorem \ref{pta}, $G_{n}$ acts on a closed Riemann surface $S$ of genus $n$ with signature $(0;2,4, 4n)$. Moreover, as this produces the smallest possible hyperbolic area \cite{May10} for a conformal action of $G_{n}$, we obtain that $\sigma^{0}(G_{n})=n$.   
Let $\Gamma$ be a Fuchsian group with signature $(0;+; [2, 4, 4n];\{-\})$ and presentation 
$\Gamma=\langle \beta_1,\beta_2,\beta_3:\ \beta_1\beta_2\beta_3=\beta_1^2=\beta_2^4=\beta_3^{4n}=1\rangle$. 
Every possible epimorphism from $\Gamma$ to $G_{n}$, with torsion-free kernel, is of the form
$$\theta_{k,r}:\Gamma\to G_{n}: 
\theta_{k,r}(\beta_1)=x^{k}y,\ \theta_{k,r}(\beta_2)=yx^{-(k+r)},\ \theta_{k,r}(\beta_3)=x^{r},$$
where $k,r\in \{0,1,\ldots,4n-1\}$ with $k$ even and ${\rm gcd}(4n,r)=1$.
By post-composing $\theta_{k,r}$ by the automorphism $\psi_{u,2v}$ of $G_{n}$, where $ur$ is congruent to $1$ module $4n$ and $ku+2v$ is congruent to 0 module $4n$, we obtain $\theta_{0,1}$.    

(2) If  $n$ is even, then this is a consequence of Theorem \ref{pta} and part (1) above.  
Let us now assume $n$ odd. Theorem \ref{pta} asserts that if $G_{n}$ acts triangular way, then it is not purely-non-free on closed Riemann surfaces of genus $n$. So, is sufficient consider a quadrilateral Fuchsian group $\Gamma$ with signature $(0;+; [2, 2, 4, 4n];\{-\})$  (this provides the smallest non-triangular hyperbolic area for the conformal action of $G_{n}$ \cite{May10}). A presentation of $\Gamma$ is
$$\Gamma=\langle \beta_1,\beta_2,\beta_3, \beta_4:\ \beta_1\beta_2\beta_3\beta_4=\beta_1^2=\beta_2^2=\beta_3^4=\beta_4^{4n}=1\rangle.$$ 
An epimorphism $\theta:\Gamma \to G_{n}$, with torsion-free kernel $K$, is given by 
$\theta(\beta_1)=y,\ \theta(\beta_2)=yx^{2},\ \theta(\beta_3)=x^n,\ \theta(\beta_4)=x^{3n-2}$. If $S=\h/K$, 
considering the branched regular cover map $\pi:S\to \widehat{\mathbb{C}}$, with deck group $C_{4n}=\langle x\rangle$, as in the proof of the item (a) of Theorem \ref{pta}, we have that $yx$ (also $yx^3)$ must has fixed points. So every element of $G_{n}$ acts with fixed points \cite[Lemma 1]{BGH}. Since $S/G_{n}$ has signature $(0;2, 2, 4, 4n)$ follows from the Riemann-Hurwitz formula that $S$ has genus equal to $3n.$   

Following \cite{BRT}, the conformal action of $G_{n}$ on closed Riemann surfaces $S$ of genus $3n$ with signature $(0;2,2,4,4n)$ is not unique up to 
homeomorphisms. Let us consider a Fuchsian group $\Gamma$ with signature $(0;+;[2,2,4,4n];\{-\})$ and an epimorphism $\theta:\Gamma\to G_n$.  It must satisfy that $\theta(\beta_1),\theta(\beta_2)\in\{x^{2n},yx^s,\ \text{$s$-even}\}$; $\theta(\beta_3)\in \{x^n, x^{3n},yx^m,\ \text{$m$-odd}\}$; $\theta(\beta_4)\in \{x^i,\ {\rm gcd}(i, 4n)=1\}$.

By post-composing $\theta$ by automorphisms of $G_n$, we obtain that all of these possibilities for $\theta$ are ${\rm Aut}(G_n)/{\rm Aut}(\Gamma)$-equivalent to
$$\theta_1:\Delta\to G_n: \theta_1(\beta_1)=y;\ \theta_1(\beta_2)=yx^2;\ \theta_1(\beta_3)=x^n;\ \theta_1(\beta_4)=x^{3n-2},\ \text{or}$$
$$\theta_2:\Delta\to G_n: \theta_2(\beta_1)=y;\ \theta_2(\beta_2)=x^{2n};\ \theta_2(\beta_3)=yx^{2n-1};\ \theta_2(\beta_4)=x,\ \text{or}$$
$$\theta_3:\Delta\to G_n: \theta_3(\beta_1)=y;\ \theta_3(\beta_2)=yx^4;\ \theta_3(\beta_3)=x^n;\ \theta_3(\beta_4)=x^{3n-4}.\qquad $$

Thereby, if $\theta$ is equivalent to $\theta_1$ it is purely-non-free (\cite[Lemma 1]{BGH}), in otherwise the action $\theta$ is not purely-non-free.
\end{proof}
    
\subsection{Jacobian variety for the triangular action of $G_{n}$} 
A conformal action of a finite group $G$ on a Riemann surface $S$ of genus $g \geq  2$ induces a natural ${\mathbb Q}$-algebra homomorphism $\rho : {\mathbb Q}[G]\to {\rm End}_{\mathbb Q}(J_S)$, from the group algebra ${\mathbb Q}[G]$ into the endomorphism algebra of the Jacobian variety $J_S$. The factorization of ${\mathbb Q}[G]$ into a product of simple algebras yields a decomposition of $J_S$ into abelian subvarieties, called the {\it isotypical decomposition} \cite{LR}.
We proceed to describe the decomposition of the Jacobian variety for the triangular action of $G_{n}$ on closed Riemann surfaces.
    
\begin{theo}\label{VJ}
Let $S$ be the closed Riemann surface of genus $n \geq 2$ admitting the conformal action of $G_{n}$, presented as in (\ref{g1}), with signature $(0; 2,4, 4n)$. Then
\setlength{\leftmargini}{8mm}
\begin{enumerate}
\item[(a)] If $n=2^{\alpha}$ with $\alpha\geq 1$, then $J_S\sim J_{S/\langle y\rangle}^2.$

\item[(b)] If $p$ is the smallest prime divisor bigger than two of $n$ and $k$ is such that $n=pk$, then
  $J_S\sim J_{S/\langle y\rangle}^2\times J_{S/\langle x^{4k}\rangle}.$ 
\end{enumerate}   

Moreover, in each case, $J_S$ has complex multiplication.
\end{theo}
\begin{proof}
\setlength{\leftmargini}{7mm}
\begin{enumerate}
\item[(\bf a)] By Kani-Rosen's Theorem \cite{KR}, applied on the subgroup $H=\langle x^{2n}, y\rangle\cong C_{2}\times C_{2}$ of $G_{n}$ and partition
  $H=\langle x^{2n}\rangle \cup \langle y \rangle \cup \langle yx^{2n} \rangle,$
the following isogeny relation holds
\begin{equation}\label{J2} 
J_{S}^{2}\times J^{4}_{S/ H} \sim J_{S/ \langle x^{2n} \rangle}^{2}\times J_{S/ \langle y \rangle}^{2}\times J_{S/ \langle yx^{2n} \rangle}^{2}.
\end{equation} 
    
 As $y$ and $yx^{2n}$ are in the same conjugacy class (Lemma \ref{pgg1}),  $J_{S/ \langle y \rangle}^{2}$ and $J_{S/ \langle yx^{2n} \rangle}^{2}$ are isogenous \cite[Proposition 2]{Pa}, so (\ref{J2}) asserts 
 \begin{equation}\label{J3}   
J_{S}^{2}\times J^{4}_{S/ H} \sim J_{S/ \langle x^{2n} \rangle}^{2}\times J_{S/ \langle y \rangle}^{4}.
 \end{equation}
 Applying Poincar\'e's complete reducibility theorem in (\ref{J3}), we obtain
 $$J_{S}\times J^{2}_{S/ H} \sim J_{S/ \langle x^{2n} \rangle}\times J_{S/ \langle y \rangle}^{2}.$$
    
 As the center of $G_{n}$ is $Z(G_{n})=\langle x^{2n}\rangle$ and $S$ is hyperelliptic, the genus of $S/ H$ and of $S/ \langle x^{2n} \rangle$ are zero.  Thereby, $J_{S} \sim J_{S/ \langle y \rangle}^{2}.$ 
    
 \item[(\bf b)] By applying Kani-Rosen's theorem on the subgroup $$H=\langle r, s \rangle=\langle x^{4k}, y : (x^{4k})^{p}=y^{2}=1, yx^{4k}y=x^{-4k} \rangle\cong D_{2p}$$ of $G_{n}$, with partition
 $H=\langle r \rangle \cup \langle s \rangle \cup \langle sr \rangle  \cup \cdots \cup \langle sr^{p-1}\rangle,$
  we obtain the isogeny relation
 \begin{equation}\label{j2} 
 J_{S}^{p}\times J^{2p}_{S/ D_{2p}} \sim J_{S/\langle r \rangle }^{p}\times J_{S/ \langle s \rangle }^{2}\times J_{S/ \langle sr \rangle }^{2} \times \cdots \times J_{S/ \langle sr^{p-1} \rangle }^{2}.
 \end{equation}
    
As the elements $s, sr, sr^{2}, \cdots, sr^{p-1}$ belong to the same conjugacy class (Lemma \ref{pgg1}),  $J_{S/ \langle s \rangle },  J_{S/ \langle sr \rangle }, \cdots, J_{S/ \langle sr^{p-1} \rangle } $ are isogenous \cite[Proposition 2]{Pa}. Then, from (\ref{j2}), we conclude that
\begin{equation}\label{j3} 
J_{S}^{p}\times J^{2p}_{S/  D_{2p}} \sim J_{S/  \langle r \rangle}^{p}\times J_{S/  \langle s \rangle}^{2p}.
\end{equation} 
     
Applying Poincar\'e's complete reducibility theorem in (\ref{j3}), permits to obtain 
$$J_{S}\times J^{2}_{S/  \langle x^{4k},y \rangle} \sim J_{S/  \langle x^{4k} \rangle}\times J_{S/  \langle y \rangle}^{2}.$$
 
 As $G_{n}$ acts in a triangular way on $S$, by results in \cite{AR} for the group $D_{2p}$, we obtain that the genus of $S/D_{2p}$ is zero. Therefore, $J_{S}\sim J_{S/  \langle x^{4k} \rangle}\times J_{S/  \langle y \rangle}^{2}.$ 
 \end{enumerate}

Finally, by \cite[Thm. 2.4.4]{Ro}, in either case (a) and (b), the Jacobian variety $J_S$ has complex multiplication. 
\end{proof}

\begin{remark}
In case (b) of the above theorem, the subgroup $\langle x\rangle$ acts on $S$ with signature $(0; 2, 4n, 4n)$.  Then, by Theorem 2.7 in \cite{ABCNPW}, for the normal subgroup $\langle x^{4k}\rangle$ of  $\langle x \rangle$ we have that the genus of the quotient orbifold $S/\langle x^{4k}\rangle$ is 1 if $n$ odd, and 2 if $n$ even. As the genus of $S$ is $n$, we conclude that the genus of the quotient orbifold $S/ \langle y \rangle$ is $(n-1)/2$ if $n$ odd, and $(n-2)/2$ if $n$ even (by dimensions arguments).
\end{remark}

\section{The symmetric hyperbolic genus of $G_{n}$}
The generalized quasi-dihedral group $G_{n}$, for $n \geq 2$, contains index two subgroups. If $H$ is any of the these subgroups, then there are closed Riemann surfaces $S$, of genus at least two, for which $G_{n} \leq {\rm Aut}(S)$ and  $H=G_{n} \cap {\rm Aut}^{+}(S)$ \cite{Gre}, we denote by  $\sigma^{hyp}(G_{n},H) \geq 2$ the smallest genus of these closed Riemann surfaces.

\begin{theo}\label{hyp1}
If $n \geq 2$, then 
\begin{enumerate} 
 \item $\sigma^{hyp}(G_{n}, D_{4n})=2n+1$.
\item $\sigma^{hyp}(G_{n}, DC_{4n})=\begin{cases}
n,&\text{ if $n\geq 2$ even},  \\
n-1, &\text{ if $n\geq 3$ odd}. 
 \end{cases}  $
 \item $\sigma^{hyp}(G_{n}, C_{4n})=2n-1$.
 \end{enumerate}
\end{theo}

The symmetric hyperbolic genus of $G_{n}$ is given by  $\sigma^{hyp}(G_{n})={\rm min}\{\sigma^{hyp}(G_{n},H):\ \text{$H<G_{n}$ and $[G_{n}:H]=2$}\}$. The above provides the following.

\begin{theo}\label{hyp}
If $n \geq 2$, then 
$$\sigma^{hyp}(G_{n})=\begin{cases}
n,&\text{ if $n\geq 2$ even},  \\
n-1, &\text{ if $n\geq 3$ odd}. 
 \end{cases} 
$$
\end{theo}

The Theorem \ref{hyp}, together \cite[Theorem 2.4]{H.S}, permits to obtain the following.

\begin{coro}\label{shg}
Every integer $g\geq 2$ is the symmetric hyperbolic genus of some finite group.
\end{coro}

\subsection{Proof of Theorem \ref{hyp1}}
Let $S$ be a closed Riemann surface of genus at least two such that $G_{n} \leq {\rm Aut}(S)$ and $G_{n}^{+}=G_{n} \cap {\rm Aut}^{+}(S)$ is any one of index two subgroups, i.e.,  $G_{n}^{+}\in \lbrace C_{4n}, D_{4n}, DC_{4n} \rbrace$. As the orders of the cyclic subgroups of $G_{n}$ are divisors of $4n$, $4$ and $2$, the orders of the conical points (if any) of the quotient orbifold $\mathcal{O}=S/G_{n}^{+}$, are also of that form.  The group $G_{n}$ induces an anticonformal involution $\tau$, on the quotient orbifold $\mathcal{O}$, so that $\mathcal{O}/ \langle \tau \rangle=S/G_{n}$. Such an involution permutes the cone points preserving the orders.

\begin{lemma}\label{lgg1}
If $G_{n}^{+}=D_{4n}$, then $\tau$ acts without fixed points.  In the other cases, $\tau$ may acts with fixed points.
\end{lemma}
\begin{proof}
If $\tau$ has fixed points, then it must have at least a simple loop (an oval) consisting of fixed points. This means that we may find a lifting of $\tau$ in $G_{n}\setminus G_{n}^{+}$ having infinitely many fixed points. The only possibility for such a lifting is to have order two, so the possibilities are $y,x^{2n},yx^{i}$ with $i$ even. In the case, $G_{n}^{+}= D_{4n}$, this is not possible since all involutions are conformal. However, if $G_{n}^{+}=C_{4n}$ or $G_{n}^{+}=DC_{4n}$,  we may consider $y, yx^{i}$ (with $i$ even), as anticonformal involutions. 
\end{proof}

\noindent{\bf Case} $G_{n}^{+}=D_{4n}$. In this case, by Lemma \ref{lgg1} we have that the number of conical points (if any) of $\mathcal{O}$ is even, say $2r$, and they are permuted in pairs by the involution $\tau$. This, in particular, asserts that $\mathcal{O}$ has signature of the form  $(h ; m_{1}, m_{1}, \ldots, m_{r}, m_{r})$ (see, \cite[Corollary 2.2.5]{BuEtaGamGro}), where $m_{j}\geqslant 2$, and $S/G_{n} = \mathcal{O}/ \langle \tau \rangle$ is a closed hyperbolic non-orientable surface, say a connected sum of $h + 1$  real projective planes and having $r$ cone points of orders $m_{1}, \ldots, m_{r}$. This means that there is an NEC group with presentation  $$\Delta=\langle d_{1}, \cdots , d_{h +1}, \beta_1,\cdots, \beta_r:\ \beta_1^{m_{1}}=\cdots=\beta_r^{m_{r}}=1,\beta_1\cdots \beta_rd_{1}^2 \cdots d_{h+1}^2=1\rangle,$$ where $d_{j}$ is a glide-reflection and $\beta_j$ is an elliptic transformation, and there is a surjective homomorphism $\theta:\Delta\to G_{n}$ such that $ \theta(\beta_j)\in G_{n}^{+}$ and $ \theta(d_j)\in G_{n}\setminus G_{n}^{+}$, with torsion-free kernel $\Gamma$ and $S = \mathbb{H}/\Gamma$. 
 By \cite{Bu03, May10}, the minimal possible genus $g\geq 2$ for $S$ is when $\mathcal{O}$ has signatures
$$i)\ (0;2,2,2n,2n),\qquad ii)\ (1;2,2),\qquad iii)\ (0;2,2,2,2,2,2).$$

In case $i)$, $\Delta$ has signature $(1;-;[2, 2n];\{-\})$ and presentation $\Delta=\langle d_1,\beta_1, \beta_2:\ \beta_1^2=\beta_2^{2n}=1, \beta_1\beta_2d_1^2=1\rangle$.
In case $ii)$, $\Delta$ has signature $(2;-;[2];\{-\})$ and presentation $\Delta=\langle d_1, d_2,\beta_1:\ \beta_1^2=1, \beta_1d_1^2d_2^2=1\rangle$. These signatures are not admissible for the action of $G_{n}$ as there exist no epimorphisms $\theta:\Delta\to G_{n}$. 
 
In case $iii)$, $\Delta$ has signature $(1;-;[2,2,2];\{-\})$ and presentation $\Delta=\langle d_{1}, \beta_1, \beta_2, \beta_3 :\ \beta_1^2=\beta_2^2= \beta_3^2=1, \beta_1\beta_2\beta_3d_{1}^2=1\rangle$. An epimorphism $\theta:\Delta\to G_{n}$, with torsion free kernel, is given by $ \theta(d_{1})=yx;\ \theta(\beta_i)=y, 1\leq i\leq 2;\ \theta(\beta_3)=x^{2n}.$ The Riemann-Hurwitz formula asserts that $S$ has genus $2n+1$.

\medskip
\noindent{\bf Case} $G_{n}^{+} \in \{C_{4n}, DC_{4n}\}$. By Lemma \ref{lgg1}, the number of conical points (if any) of $\mathcal{O}$ could be odd, say $2r+t$, where the $2r$ points are permuted in pairs by the involution $\tau$, and the $t$ points are associated to the fixed points of  $\tau$ (the ovals). This, asserts that $\mathcal{O}$ has signature of the form \cite[Corollary 2.2.5]{BuEtaGamGro}: $$(h ; m_{1}, m_{1}, \ldots, m_{r}, m_{r}, n_{11},\ldots, n_{1s_{1}},\cdots, n_{k 1},\ldots, n_{k s_{k}}),$$ where $m_{j}, n_{l s_{l}}\geqslant 2$, $s_1+\cdots+s_k=t$, and $S/G_{n}= \mathcal{O}/ \langle \tau \rangle$ is a bordered surface. By \cite{May10}, a minimal genus $g\geq 2$ for $S$ is obtained when $\mathcal{O}$ has triangular signature $(0;m_1,m_2, m_3)$. So, there is a Fuchsian group $\Gamma$ with signature $(0;+; [m_{1}, m_{2}, m_{3}];\{-\})$ as the canonical Fuchsian group of $\Delta$.
So, there are three possible signatures to consider for $\Delta$ \cite{Bu0}:
$$i)\ (0;+; [-];\lbrace ( m_{1}, m_{2}, m_{3})\rbrace),\quad ii)\ (0;+; [-];\lbrace ( m_{1}, m_{1}, m_{3})\rbrace),\quad iii)\ (0;+; [m_{1}];\lbrace ( m_{2})\rbrace).$$

\medskip
\noindent{\bf Cases} $i$) {and} $ii)$. In these cases, $\Delta$ has the following presentation
$$\Delta=\langle c_{10}, c_{1 1}, c_{1 2}, c_{1 3}, e_{1}:\ c_{1j}^{2}=(c_{1 j-1}c_{1j})^{m_{j}}=1, e_{1}^{-1}c_{1 0}e_{1}c_{1 3}= e_{1}=1\rangle.$$
There exist no an epimorphism $\theta: \Delta \to G_{n}$. In fact, if  such an epimorphism $\theta$ exists, then it has to preserve the relations of $\Delta$, in particular $(\theta(c_{1j}))^{2}=1$ and  $\theta(e_{1})=1$. By Lemma \ref{pgg1}, $\theta(c_{1j})\in \lbrace y, yx^{2}, yx^{i_j},\ \text{$i_j$-even} \rbrace $, and these elements do not generate the group $G_{n}$.

\medskip
 \noindent{\bf Case} $iii$). Let us first assume $G_{n}^{+}= DC_{4n}$. By results in \cite{H.S}, a minimal genus $g\geq 2$ for $S$ is obtained when $\mathcal{O}$ has either signature (a) $(0;4,4, 2n)$, for $n$ even or (b) $(0;4,4,n)$, for $n$ odd. So, $\Delta$ has signature $(0;+;[4];\{(m_3)\}$ and presentation
$$\Delta=\langle \beta_1, c_{10}, c_{1 1}, e_{1}:\ \beta_1^{4}=c_{1j}^{2}=(c_{1 0}c_{11})^{m_{3}}=1, e_{1}^{-1}c_{1 0}e_{1}c_{1 1}=1,\beta_1e_{1}=1\rangle.$$ 

An epimorphism $\theta:\Delta\to G_{n}$ can be defined as:
$$\theta(e_{1})=yx^{2n+1}, \, \theta(\beta_1)=yx, \, \theta(c_{10})=y,\,  \theta(c_{11})=yx^{2n+2}\quad (\text{for all $n$}).$$

By the Riemann-Hurwitz formula, we obtain that $g=n$ for $n$ even, and $g=n-1$ for $n$ odd.
 
 Let us now assume $G_{n}^{+}= C_{4n}$. By results in \cite{Ha}, a minimal genus $g\geq 2$ for $S$ is possible when $\mathcal{O}$ has signature $(0;4n, 4n, t)$ with $t \geq 2$ dividing $4n$. So, $\Delta$ has signature $(0; +;[4n]; \{(t)\}$ and presentation
 $$\Delta=\langle \beta_1, c_{10},c_{11},e_1:\ \beta_1^{4n}=c_{1j}^2=(c_{10}c_{11})^t=1, e_1c_{10}e_1^{-1}c_{11}=\beta_1e_1=1\rangle.$$
 
 We want to describe an epimorphism $\theta:\Delta\to G_{n}$ such that $\theta(\Delta^{+})=G_{n}^{+}$ and with torsion-free kernel, where
 $$\theta(\beta_{1})\in \{x^{i}, \text{{\small gcd}$(i, 4n)=1$}\}, \theta(e_1)=\theta(\beta_1)^{-1}, \theta(c_{10}) \in \lbrace y, yx^s, \text{$s$-even}\}, \theta(c_{11}) \in \lbrace  y, yx^r, \text{$r$-even}\},\ $$  $$(\star)\quad (\theta(c_{10})\theta(c_{11}))^{t}=1,\ \theta(e_1)\theta(c_{10})\theta(e_1)^{-1}\theta(c_{11})=1.$$
We may take $\theta(\beta_1)=x^{-1}$, $\theta(c_{10})=y$ and $\theta(c_{11})=yx^r$. Then from the equality  $(\star)$ we have
$$-2n+2+r\equiv 0\ {\rm mod}(4n)\quad \text{and}\quad tr\equiv 0\ {\rm mod}(4n).$$ So, 
$r=2n-2$ and $t=2n$.
 Therefore, an epimorphism $\theta:\Delta\to G_{n}$ is given by
 $$\theta(\beta_1)=x^{-1},\ \theta(e_1)=x,\ \theta(c_{10})=y,\ \theta(c_{11})=yx^{2n-2}.$$
  By the Riemann-Hurwitz formula, we get $g = 2n-1.$
 
 In summary, the symmetric hyperbolic genus of $G_{n}$ is equal to $n$ if $n$ even, and $n-1$ if $n$ odd, with $G_{n}^{+}$ equal to  $DC_{4n}$.  
 \qed

\subsection{Uniqueness on the minimal symmetric hyperbolic genus}
\begin{theo}\label{tumshg}
Let $n \geq 2$ and $H$ be an index two subgroup of $G_{n}$. 
Then the  action of $G_{n}$ (admitting anticonformal elements and such that $H$ is its conformal part) on the symmetric hyperbolic genus $\sigma^{hyp}(G_{n},H)$ is unique (up to homeomorphisms). In particular, the  action of $G_{n}$ (admitting anticonformal elements) on the symmetric hyperbolic genus $\sigma^{hyp}(G_{n})$ is unique (up to homeomorphisms).
\end{theo}
\begin{proof} Let us consider the presentation of $G_{n}=\langle x,y\rangle$ as in \eqref{g1}. We have three index two subgroups $H \in \{D_{4n}, DC_{4n},C_{4n}\}$ of $G_n$. 
Let $S/G_{n}=\h/\Delta$, where $\Delta$ is an NEC group,  and  $S/H=\h/\Delta^{+}$, where $\Delta^{+}$
is the corresponding canonical Fuchsian group. 
We need to prove that, up to pre-composition by automorphisms of $\Delta$ and post-composition by automorphisms of $G_{n}$, there is exactly one epimorphism $\theta:\Delta\to G_{n}$ with torsion-free kernel such that $\theta(\Delta^{+})=H$. 

\medskip
\noindent{\bf Case $H= D_{4n}$}.
In this case, as a consequence of the proof of Theorem \ref{hyp1}, the quotient orbifold $S/G_{n}$ has signature  $(1;-;[2,2,2];\{-\})$; so $\Delta$ has a presentation 
$\Delta=\langle d_1,\beta_1,\beta_2, \beta_3:\ \beta_1^2=\beta_2^2=\beta_3^2=1, \beta_1\beta_2\beta_3d_1^2=1\rangle$  and
$\Delta^{+}=\langle d_1^2,\beta_1,\beta_2, \beta_3, d_1\beta_1d_1^{-1}, d_1\beta_2d_1^{-1}, d_1\beta_3d_1^{-1}\rangle.$ 
The condition for $\theta$ to have torsion-free kernel ensures that $\theta(\beta_{j})$ is an order two element of $G_{n}$, that is,
$\theta(\beta_{j}) \in \lbrace x^{2n}, x^{2s_{j}}y, \, s_{j}\in \lbrace0, 1, \ldots , 2n-1 \rbrace \rbrace$.
As $\theta$ is surjective and $\theta(\Delta^{+})=H$, we must also have $\theta(d_{1})=x^{m}y^{t}$, where $m$ is odd and $t\in \{ 0, 1\}$.
As $\theta$ is a homomorphism, we must also have the relation 
$$(*) \quad \theta(\beta_{1})\theta(\beta_{2}) \theta(\beta_{3})\theta(d_{1})^{2}=1.$$

If $t=0$, then (up to post-composition by an automorphism of $G_{n}$ of the form $\psi_{l,0}$), we may assume $\theta(d_{1})=x$. Now, the equality $(*)$ asserts that one of the values of $\theta(\beta_{j})$ must be $x^{2n}$ and the others two must be of the form $x^{2s_{j}}y$. So, up to pre-composition by a power of the automorphism of $\Delta$ defined by $(d_{1},\beta_{1},\beta_{2},\beta_{3}) \mapsto (\beta_{3}d_{1}\beta_{3},\beta_{3},\beta_{1},\beta_{2})$, we may assume that $\theta(\beta_{3})=x^{2n}$, $\theta(\beta_{j})=x^{2s_{j}}y$, for $j=1,2$. In this way, by applying the equality $(*)$, we may observe that 
$$(s_{1},s_{2}) \in \{(n-1,0), (n,1),(n+1,2),\ldots,(2n-1,n),(0,n+1),(1,n+2),\ldots,(n-2,2n-1)\}.$$

All of these possibilities are equivalent (up to post-composition by some $\psi_{l,0}$) to
$$\theta_1:\Delta\to G_{n}:\ \theta_1(d_1)=x;\ \theta_1(\beta_3)=x^{2n};\ \theta(\beta_2)=y;\ \theta(\beta_1)=x^{2n-2}y.$$ 

If $t=1$, then (up to post-composition by an automorphism of $G_{n}$ of the form $\psi_{l,0}$), we may assume $\theta(d_{1})=xy$. Proceeding in a similar fashion as in the previous case, we may obtain that $\theta$ is equivalent to 
$$\theta_2:\Delta\to G_{n}:\ \theta_2(d_1)=xy;\ \theta_2(\beta_3)=x^{2n};\, \theta_{2}(\beta_2)=y;\ \theta_{2}(\beta_1)=y.$$

Moreover, if $L$ is the automorphism of $\Delta$ given by $$L(d_{1})=\beta_1d_1^{-1}; \, L(\beta_1)=d_1^{-1}\beta_3d_1; \ L(\beta_2)=d_1^{-1}\beta_{2}d_{1}; \, L(\beta_3)=\beta_{1},$$ we note that $\widehat{\theta}=\psi_{4n-1,0} \circ \theta_1\circ L:\Delta\to G_{n}$ is given by 
$$ \widehat{\theta}=(d_1)=xy;\ \widehat{\theta}=(\beta_3)=x^{2(n+1)}y;\, \widehat{\theta}=(\beta_2)=x^{2(n+1)}y;\ \widehat{\theta}=(\beta_1)=x^{2n},$$
which corresponds to the case $t=1$, so equivalent to $\theta_2$.  

\medskip
\noindent{\bf Case $H= DC_{4n}$}. 
By the proof of Theorem \ref{hyp1}, the quotient orbifold $S/G_{n}$ has signature  $(0;+;[4];\{(2n)\})$ for $n$ even and $(0;+;[4];\{(n)\})$ for $n$ odd. So, 
$$\Delta=\langle \beta_1, c_{10}, c_{11}, e_1:\ \beta_1^4=c_{1j}^2=(c_{10}c_{11})^{m_3}=1, e_1c_{10}e_1^{-1}c_{11}=\beta_1e_1=1\rangle.$$ 
 In this case, 
$\theta(\beta_{1})\in \{x^{n}, x^{3n},yx^{m}, \text{$m$-odd}\}\ \text{for $n$ even};\ \theta(\beta_{1})\in \{yx^{m}, \text{$m$-odd}\}\ \text{for $n$ {odd}};\ \theta(e_1)=\theta(\beta_1)^{-1};$
$\theta(c_{10}) \in \lbrace  y, yx^s, \text{$s$-even}\};\ \theta(c_{11}) \in \lbrace  y, yx^r, \text{$r$-even}\},$ and
$(\theta(c_{10})\theta(c_{11}))^{2n}=1\ \text{for $n$ even};$ $(\theta(c_{10})\theta(c_{11}))^{n}=1\ \text{for $n$ odd};$ $\theta(e_1)\theta(c_{10})\theta(e_1)^{-1}\theta(c_{11})=1.$

All of these possibilities are equivalent (up to post-composition by some automorphisms of $G_n$ of the form $\psi_{1,2v}$ and $\psi_{l,0}$) to
$$\theta_1:\Delta\to G_{n}:\ \theta_1(e_{1})=yx^{2n+1}; \, \theta_1(\beta_1)=yx; \, \theta_1(c_{10})=y;\,  \theta_1(c_{11})=yx^{2n+2}.$$

\medskip
\noindent{\bf Case $H= C_{4n}$}. 
By the proof of Theorem \ref{hyp1}, the quotient orbifold $S/G_{n}$ has signature $(0;+;[4n];\{(2n)\})$. So, 
$\Delta=\langle \beta_1, c_{10}, c_{11}, e_1:\ \beta_1^{4n}=c_{1j}^2=(c_{10}c_{11})^{2n}=1,e_1c_{10}e_1^{-1}c_{11}=\beta_1e_1=1\rangle.$ 
In this case, 
$\theta(\beta_{1})\in \{x^{i}, \text{${\rm gcd}(i, 4n)=1$}\};$ $\theta(e_1)=\theta(\beta_1)^{-1};\ \theta(c_{10}) \in \lbrace  y, yx^s, \text{$s$-even}\};$ $\theta(c_{11}) \in \lbrace  y, yx^r, \text{$r$-even}\},$ $(\theta(c_{10})\theta(c_{11}))^{2n}=1$ and $\theta(e_1)\theta(c_{10})\theta(e_1)^{-1}\theta(c_{11})=1.$

All of these possibilities are equivalent (up to post-composition by some automorphisms of $G_n$ of the form $\psi_{l,0}$ and $\psi_{1,2v}$) to
$$\theta_1:\Delta\to G_{n}:\ \theta_1(\beta_1)=x^{-1};\ \theta_1(e_1)=x;\ \theta_1(c_{10})=y;\ \theta_1(c_{11})=yx^{2n-2}.$$

\end{proof}

\section{Pseudo-real actions of the group $G_{n}$}
In this section, we consider conformal/anticonformal actions of the group $G_{n}=\langle x, y\rangle$ (with presentation as in \eqref{g1}) on pseudo-real Riemann surfaces. There are two cases to consider: either $G_{n}$ has anticonformal elements or $G_{n}$ only contains conformal elements.

\subsection{Conformal/anticonformal actions of $G_{n}$ on pseudo-real Riemann surfaces}
In this section, we look for pseudo-real Riemann surfaces $S$ with $G_{n} \leq {\rm Aut}(S)$ and $G_{n} \neq G_{n} \cap {\rm Aut}^{+}(S)$.  

\begin{lemma}\label{ABQ}
The number of involutions of the full group of automorphisms of a pseudo-real Riemann surface is equal to the number of involutions of their group of conformal automorphisms. 
 \end{lemma}
  \begin{proof}
 Let $S$ be a pseudo-real Riemann surface with full group of automorphisms ${\rm Aut}(S)$ and group of conformal automorphisms ${\rm Aut}^{+}(S)$. 
 Let $n$ and $m$ be the number of involutions of the group ${\rm Aut}^{+}(S)$ and ${\rm Aut}(S)$ respectively.  Suppose that $n\not=m$, i.e., either $n<m$ or $n>m$. The case $n>m$ is not possible by that ${\rm Aut}^{+}(S)$ is a subgroup of ${\rm Aut}(S)$. Now, if $n<m$, then there is an anticonformal involution in ${\rm Aut}(S)$, it's no possible by that $S$ is a pseudo-real Riemann surface.
 \end{proof}
 
As the only index two subgroup of $G_{n}$ containing all the involutions of $G_{n}$ is $D_{4n}=\langle x^{2},y\rangle$, we must have $G_{n} \cap {\rm Aut}^{+}(S)=D_{4n}$.

\subsubsection{\bf Construction of pseudo-real Riemann surfaces}
Next, we construct pseudo-real Riemann surfaces $S$ with $G_{n} \leq {\rm Aut}(S)$ and $G_{n} \cap {\rm Aut}^{+}(S)=D_{4n}$.

\begin{theo}\label{tps}
Let $k\geq 3$ and $n\geq 2$ be integers. Then there are pseudo-real Riemann surfaces $S$ of genus $g=2nk-4n+1$ such that $G_{n} \leq {\rm Aut}(S)$ and $G_{n} \cap {\rm Aut}^{+}(S)=D_{4n}$. Moreover, these surfaces can be assumed to satisfy $G_{n}={\rm Aut}(S)$.
\end{theo}

\begin{proof}
Let $k\geq 3$ and  we consider the NEC group 
$$\Delta=\langle d_{1}, \beta_1,\cdots, \beta_k:\ \beta_1^2=\cdots=\beta_k^2=1,\beta_1\cdots \beta_kd_{1}^2=1\rangle.$$ Then the quotient Klein surface uniformized by $\Delta$ is the orbifold whose underlying surface is the real projective plane and its conical points are $k$ points, each one of order $2$. Let us consider the epimorphism $\theta:\Delta\to G_{n}$ given by (with $a=1$ if $k$ even and $a=0$ if $k$ odd):
$$\theta(d_{1})=yx; \, \theta(\beta_i)=y, 1\leq i\leq k-1; \, \theta(\beta_k)=y^ax^{2n}.$$ 

The kernel $\Gamma$  of $\theta$ is a torsion-free subgroup contained in the half-orientation part $\Delta^{+}$ of $\Delta$. The Riemann surface $S=\mathbb{H}/\Gamma$ is a closed Riemann surface of genus $g\geq 2$ with $G_{n} \leq {\rm Aut}(S)$ and such that $G_{n} \cap {\rm Aut}^{+}(S)=D_{4n}$.

Since the signature of the quotient orbifold $S/D_{4n}$  is $ (0; \stackrel{2k}{2,\cdots, 2})$, where the number of cone points is exactly $2k\geq 6$, it follows from Singerman list of maximal Fuchsian groups \cite{S.}, that we may choose $\Delta$ so that ${\rm Aut}(S)=G_{n}$. In this case, as the only anticonformal  automorphisms of $S$ are the elements of $G_{n} \setminus D_{4n}$ (each one of order different than two, Lemma \ref{pgg1}), it follows that $S$ is a pseudo-real Riemann surface. The Riemann-Hurwitz formula, applied to the branched cover $S \rightarrow S/ D_{4n}$, permits to see that $S$ has genus $g=2nk-4n+1.$
\end{proof}

\subsubsection{\bf Minimal pseudo-real genus for $G_{n}$}
Above we have constructed pseudo-real Riemann surfaces $S$ such that $G_{n} \leq {\rm Aut}(S)$ and $G_{n} \cap {\rm Aut}^{+}(S)=D_{4n}$.

The following result was proved in \cite[Proposition 3.2]{CL}. 

\begin{prop}\label{pmprs} Let $n \geq 2$.
The minimal genus of a closed pseudo-real Riemann surface $S$ such that $G_{n} \leq {\rm Aut}(S)$ and $D_{4n}=G_{n} \cap {\rm Aut}^{+}(S)$ is $2n+1$. 
In that case, the signature of the quotient orbifold $S/D_{4n}=\h/\Delta^{+}$ is  $(0;+;[2,2,2,2,2,2];\{-\})$ and the signature of the quotient orbifold $S/G_{n}=\h/\Delta$ is $(1;-;[2,2,2];\{-\})$. 
\end{prop}

As a consequence of the Theorem \ref{tumshg} (for the case $G_{n}^{+}=D_{4n}$) and the Proposition \ref{pmprs}  we have the following.

\begin{coro}\label{Ups} Up to homomorphisms, there is a unique action of $G_{n}$ as a group of conformal/anticonformal automorphisms on closed pseudo-real Riemann surfaces of genus $2n+1$ such that $G_{n} \cap {\rm Aut}^{+}(S) = D_{4n}$.
\end{coro}

\subsubsection{\bf Searching for equations.} 
Let $S$ be a pseudo-real Riemann surface of genus $2n+1$, for $n\geq 2$, such that $G_{n}\leq {\rm Aut}(S)$ and $D_{4n}=G_{n}\cap{\rm Aut}^{+}(S)$. 
Let $P:S\to S/\langle x^{2n},y\rangle$ be a branched regular cover map with deck group $\langle x^{2n},y\rangle \cong C_{2}^{2}$.

The Proposition \ref{pmprs} asserts that $S/D_{4n}$ has signature $(0;+;[2,2,2,2,2,2];\{-\})$. This signature implies that:\  
${\rm Fix}(x^2)={\rm Fix}(x^4)=\cdots={\rm Fix}(x^{2n-2})=\emptyset;\ \#{\rm Fix}(x^{2n})=2nF.$

\begin{lemma}\label{lfix}  
In the above situation, 
\begin{itemize}
\item $\#{\rm Fix}(x^{2n})=4n$;
\item $\#{\rm Fix}(x^{2s}y)=4$, $s=0,1,\cdots, 2n-1$;
\item The signature of $S/\langle x^{2n}, y\rangle$ is $(0;+;[\stackrel{2n+4}{2,\cdots, 2}];\{-\})$.
\end{itemize}
\end{lemma}

\begin{proof} Suppose $\#{\rm Fix}(y)=2E_1,\ \#{\rm Fix}(x^{2}y)=2E_2,\ \#{\rm Fix}(x^{2n})=2nF$ (if $n\geq 2$ even, we already know that $E_1=E_2$). But, by looking at the ``unique" homomorphism $\theta:\Delta\to G_{n}$ (Corollary \ref{Ups}) given by $\theta(d_1)=x$; $\theta(a_3)=x^{2n}$; $\theta(a_2)=y$; $\theta(a_1)=x^{2n-2}y$, we have $E_1=E_2=E$ and $F>0$.

The quotient orbifold $\mathcal{O}=S/\langle x^{2n},y\rangle$ has signature $(\gamma;+;[\stackrel{\alpha}{2,\cdots, 2}];\{-\})$, and since, ${\rm Fix}(x^{2l})=\emptyset$ with $l=1,\cdots, n-1$, $\#{\rm Fix}(x^{2n})=2nF$ and $\#{\rm Fix}(x^{2s}y)=2E$ we obtain \begin{equation}\label{leq1} \alpha=2E+nF.\end{equation} On the other hand by the Riemann-Hurwitz formula we have $2n+1=4(\gamma+1)+1+\alpha,$ so
\begin{equation}\label{leq2} 2n=4(\gamma-1)+\alpha.\end{equation}

The automorphism $x^2$ induces an automorphism $\widehat{x^2}$ of order $n$ in $\mathcal{O}$, thereby we can consider the covering $Q:\mathcal{O}\to \widehat{\mathbb{C}}$ such that $\mathcal{O}/\langle \widehat{x^2} \rangle\cong S/D_{4n}$. So, the covering $Q\circ P$ and (\ref{leq1}) implies that 
\begin{equation}\label{leq3} 6=2E+F.\end{equation} Then substituting the equations (\ref{leq1}) and (\ref{leq3}) in the equation (\ref{leq2}) we have
\begin{equation}\label{leq4} (2-F)(n-1)=4\gamma
\end{equation} that implies $F\in \{1,2\}$. 

\medskip
\noindent{\bf Case $\gamma=0$}. From the equations (\ref{leq4}), (\ref{leq3}) and (\ref{leq2}) we obtain that $F=2$, $E=2$ and $\alpha=2n+4$.

\medskip
\noindent{\bf Case $\gamma\geq 1$}. From the equation (\ref{leq4}) we obtain $F=1$ and $\gamma=\frac{n-1}{4}$. So $\#{\rm Fix}(x^{2n})=2n$.
Let us consider the branched regular cover map $S\to R=S/\langle x^2\rangle$, with deck group $\langle x^2\rangle$. In this case, the quotient orbifold $R$ has signature $(h;+;[2,2];\{-\})$. By the Riemann-Hurwitz formula,  $2n+1=2n(h-1)+1+\frac{1}{2}(2\cdot n)$, this implies that $2=2(n-1)+1$ (a contradiction).
\end{proof}

For what follows, we assume $n\geq 2$ even, and we consider the subgroups $H_0=\langle x^{2n}, y\rangle\cong C_2^2$ and $x^lH_1x^{-l}=\langle x^{2n},x^{2n+2l}y\rangle=\langle x^{2n},x^{2l}y\rangle=H_l$ with $l=1,\cdots, 2n-1$.

Let the map $S\to S/\langle x^{2n}\rangle$ be a branched regular with deck group $\langle x^{2n}\rangle$. By Lemma \ref{lfix}, together with the Riemann-Hurwitz formula, the quotient orbifold $S/\langle x^{2n}\rangle$ has signature $(1;+;[\stackrel{4n}{2,\cdots,2}];\{-\})$. 

Also, we consider branched regular cover maps $\pi_1:S\to S/\langle y\rangle$, with deck group $\langle y\rangle$, and $\pi_2:S\to S/\langle x^{2n}y\rangle$, with deck group $\langle x^{2n}y\rangle$, where the quotient orbifolds $\mathcal{O}_1=S/\langle y\rangle$  and $\mathcal{O}_2=S/\langle x^{2n}y\rangle$ has signature $(n;+;[{2,2,2, 2}];\{-\})$ respectively.

The automorphism $x^{2n}y$ induces an involution $\tau_1$ on $\mathcal{O}_1$, where $\tau_1$ has $2n+2$ fix points and permutes the four branch points of $\pi_1$, and $\mathcal{O}_1/\langle \tau_1\rangle\cong S/H_0$. Similarly, $y$ induces an involution $\tau_2$ on $\mathcal{O}_2$, where $\tau_2$ has $2n+2$ fix points and permutes the four branch points of $\pi_2$, and $\mathcal{O}_2/\langle \tau_2\rangle\cong S/H_0$.

Let $\psi_1:\mathcal{O}_1\to \mathcal{O}_1/\langle \tau_1\rangle$ and $\psi_2:\mathcal{O}_2\to \mathcal{O}_2/\langle \tau_2\rangle$ be Galois branched coverings, with deck groups $ \langle \tau_1\rangle$ and $\langle \tau_1\rangle$ respectively, such that $P=\psi_{j} \circ \pi_{j}$.

If $B_{\pi_1}=\{p_1,p_2,p_3,p_4\}$ is the set of branch points of $\pi_1$, then $\psi_1(B_{\pi_1})=\{p_1',p_2'\}$. Similarly, if $B_{\pi_2}=\{q_1,q_2,q_3,q_4\}$ is the set of branch points of $\pi_2$, then  $\psi_2(B_{\pi_2})=\{q_1',q_2'\}$. So, the set of branch values of $\psi_1$ and $\psi_2$ are $B_{\psi_1}=\{z_1,\cdots,z_{2n},p_1',p_2',q_1',q_2'\}=B_{\psi_2}$.

The automorphism $x^n$ (of order 4) induces a M\"obius transformation $j:\widehat{\mathbb{C}}\to \widehat{\mathbb{C}}$ of order two which permutes the set $\{q_1',q_2'\}$ with the set $\{p_1',p_2'\}$, and permutes the set $\{z_1,\cdots,z_{2n}\}$ (none of them being fixed by $j$). We may assume (up a M\"obius transformation) that $j(t)=-t$; $j(z_1)=1$ and $j(z_2)=-1$; $j(p_1')=\lambda_n$ and $j(q_1')=-\lambda_n$; $j(p_2')=\lambda_{n+1}$ and $j(q_2')=-\lambda_{n+1}$. So, we have that 
$$j(B_{\psi_1})=\{\pm1, \pm \lambda_1,\pm \lambda_2,\cdots, \pm \lambda_{n-1},-\lambda_n,\lambda_n,\lambda_{n+1},-\lambda_{n+1}\}$$ 
(these values have some restriction given by the anticonformal automorphism $x$). 

It follows that, if  
\begin{center}$S_1: w_1^2=(t^2-1)(t+\lambda_n)(t+\lambda_{n+1})\prod_{j=1}^{n-1}(t^2-\lambda_j^2),$
\end{center}
\begin{center}$S_2: w_1^2=(t^2-1)(t-\lambda_n)(t-\lambda_{n+1})\prod_{j=1}^{n-1}(t^2-\lambda_j^2),$
\end{center} 
then $\mathcal{O}_1$ (respectively, $\mathcal{O}_2$) has underlying Riemann surface given by the algebraic curve $S_1$ (respectively, $S_2$). 

\medskip

As consequence of the above, together Kani-Rosen's theorem \cite{KR} on the group $\langle x^{2n}, y\rangle$, we obtain the following. 

\begin{theo}\label{Epsr} Let $n\geq 2$ be an even integer. If $S$ is a pseudo-real Riemann surface of genus $2n+1$ such that $G_{n} \leq {\rm Aut}(S)$, and $D_{4n}=G_{n} \cap {\rm Aut}^{+}(S)$ (so $S/G_{n}$ has signature $(1;-; [2,2,2];\{-\})$), then
\setlength{\leftmargini}{8mm}
\begin{enumerate}
\item $S_1$ is isomorphic to $S_2$;
\item $J_S\sim J_{S/\langle y\rangle}^2\times J_{S/\langle x^{2n}\rangle}$;
\item $S:\begin{cases}
w_1^2=(t^2-1)(t+\lambda_n)(t+\lambda_{n+1})\prod_{j=1}^{n-1}(t^2-\lambda_j^2),\\
w_1^2=(t^2-1)(t-\lambda_n)(t-\lambda_{n+1})\prod_{j=1}^{n-1}(t^2-\lambda_j^2),
\end{cases}$ 

\noindent
where $y(t,w_1,w_2)=(t,w_1,-w_2)$ and $x^n(t,w_1,w_2)=(-t,w_2,-w_1)$.
\end{enumerate}
\end{theo}

\begin{remark} 
Unfortunately, in the above algebraic model for $S$ we cannot see the anticonformal automorphism $x$. Below we consider the case $n=2$.
\end{remark}

\subsubsection{\bf Family of non-hyperelliptic pseudo-real curves of genus $5$} 
 In this section, we consider the generalized quasi-dihedral group $G_{2}=\langle x, y:\ x^8=y^2=1,yxy=x^3\rangle$ of order 16 and  $D_{8}=\langle x^2, y\rangle$.
By Lemma \ref{lfix}, the quotient orbifold $S/H_1$, where $H_1=\langle x^4, y\rangle$, has signature $(0;2,2,2,2,2,2,2,2)$. By Theorem \ref{Epsr}, an algebraic model of $S$, reflecting the action of $D_{8}$, is 

\begin{center}$\mathcal{C}_1:\begin{cases}
s_1^2=(t^2-1)(t^2-\lambda_1^2)(t+\lambda_2)(t+\lambda_3),\\
s_2^2=(t^2-1)(t^2-\lambda_1^2)(t-\lambda_2)(t-\lambda_3),
\end{cases}$\end{center}
where the automorphisms $y$ and $x^{2}$ correspond to 
$y(t, s_1, s_2)=(t, s_1, -s_2)$ and $x^2(t, s_1, s_2)=(-t, s_2, -s_1)$.
The automorphism $x$ acts on $\mathcal{C}_1$ and induces an imaginary reflection $\rho$ on the quotient surface $\mathcal{C}_1/D_{8}$, where $\rho(u)=\frac{\lambda_1^2}{\overline{u}}$, $\lambda_1=i\alpha$ with $\alpha\in \mathbb{R}\setminus \{0\}$, $\lambda_3=\frac{\lambda_1}{\overline{\lambda}_2}$. So,
$\rho(0)=\infty$, $\rho(1)=\lambda_1^2$ and $\rho(\lambda_2^2)=\lambda_3^2$.
Inside $D_{8}$ there is another subgroup isomorphic to the group $C_2^2$, this given by $H_2=\langle x^4, x^2y\rangle$. Let us observe that $xH_1x^{-1}=H_2.$ Now, in the same way as for $H_1$, by Theorem \ref{Epsr}, we obtain another model for $S$:
\begin{center}$\mathcal{C}_2:\begin{cases}
s_3^2=(t_1^2-1)(t_1^2-\lambda_1^2)(t_1-\lambda_2)(t_1+\lambda_3),\\
s_4^2=(t_1^2-1)(t_1^2-\lambda_1^2)(t_1+\lambda_2)(t_1-\lambda_3),
\end{cases}$
\end{center} 
where the automorphisms $x^{2}$ and $y$ are given by
$x^2(t_1, s_3, s_4)=(-t_1, s_4, -s_3)$ and $y(t_1, s_3, s_4)=(-t_1,- s_4, -s_3)$.
Now, the automorphism $x$ induces an anticonformal isomorphic $\widehat{x}$  of $\mathcal{C}_1$ to $\mathcal{C}_2$ given by $\widehat{x}(t, s_1, s_2)=(\frac{\lambda_1}{\bar{t}}, \frac{\epsilon \lambda \bar{s}_1}{\bar{t}^3}, \frac{\rho \lambda \bar{s}_2}{\bar{t}^3})$, with $\lambda=\sqrt{-\lambda_1^2\lambda_2\lambda_3}$.
Let
\begin{center}$\overline{\mathcal{C}}_2:\begin{cases}
s_3^2=(t_1^2-1)(t_1^2-\lambda_1^2)(t_1-\bar{\lambda}_2)(t_1+\bar{\lambda}_3),\\
s_4^2=(t_1^2-1)(t_1^2-\lambda_1^2)(t_1+\bar{\lambda}_2)(t_1-\bar{\lambda}_3).
\end{cases}$
\end{center}

An anticonformal isomorphism between $\mathcal{C}_{2}$ and $\overline{\mathcal{C}}_{2}$ is given by  
$$\varphi:\mathcal{C}_2\to \overline{\mathcal{C}}_{2}: \varphi(t_1, s_3, s_4)=(\bar{t}_1, \bar{s}_3, \bar{s}_4)$$ 
and a conformal one is given by
$$\varphi\circ \widehat{x}:\mathcal{C}_1\to \overline{\mathcal{C}}_{2}:(\varphi\circ \widehat{x})(t, s_1, s_2)=\left(\frac{-\lambda_1}{t}, \frac{\epsilon \bar{\lambda}s_1}{t^3}, \frac{\rho \bar{\lambda} s_2}{t^3}\right).$$

The holomorphic map 
$$\Theta:\mathcal{C}_1\to \mathcal{C}_1\times \overline{\mathcal{C}}_{2}:\Theta(t, s_1, s_2)=\left(t, s_1, s_2, \frac{-\lambda_1}{t}, \frac{\epsilon \bar{\lambda}s_1}{t^3}, \frac{\rho \bar{\lambda}s_2}{t^3}\right)$$
induces a biholomorphism between $\mathcal{C}_1$ and $\Theta(\mathcal{C}_1)$ (the inverse map is given by the projection $(a, b) \in \Theta(\mathcal{C}_1)\mapsto a \in \mathcal{C}_1$). The equations for $\Theta(\mathcal{C}_{1})$ are

\begin{center}$\Theta(\mathcal{C}_1):\begin{cases}
s_1^2=(t^2-1)(t^2-\lambda_1^2)(t+\lambda_2)(t+\lambda_3)\\
s_2^2=(t^2-1)(t^2-\lambda_1^2)(t-\lambda_2)(t-\lambda_3)\\
s_3^2=(t_1^2-1)(t_1^2-\lambda_1^2)(t_1-\bar{\lambda}_2)(t_1+\bar{\lambda}_3)\\
s_4^2=(t_1^2-1)(t_1^2-\lambda_1^2)(t_1+\bar{\lambda}_2)(t_1-\bar{\lambda}_3)\\
tt_1=-\lambda_1\\
t^3s_3=\epsilon \bar{\lambda}s_1\\
t^3s_4=\rho\bar{\lambda}s_2
\end{cases}$
\end{center} 

In this algebraic model, 
$$y(t,s_1, s_2, t_1, s_3, s_4)=(t, s_1, -s_2, t_1, s_3, -s_4),$$ 
$$x(t,s_1, s_2, t_1, s_3, s_4)=(\bar{t}_1, \bar{s}_3, \bar{s}_4, -\bar{t}, \bar{s}_2, -\bar{s}_1).$$

\subsection{Conformal actions of $G_{n}$ on pseudo-real Riemann surfaces}
In this section, we study pseudo-real Riemann surfaces with $G_{n}$ as a group of conformal automorphisms.

\subsubsection{\bf An example}
Let $G_n$ as in (\ref{g1}). The group  $K_{n}=G_n\times C_4=\langle x, y\rangle\times \langle z\rangle$ is a non-abelian group of order $32n$, containing 
$H_{n}=G_n\times \langle z^2\rangle$ as index two subgroup. Both, $K_{n}$ and $H_{n}$, have the same number $4n+3$ of involutions.

 \begin{theo} \label{ejemplo}
 Let $\alpha\geq 1$, $n\geq 2$ be integers and $\beta\geq 2$ be an even integer. There are closed pseudo-real Riemann surfaces $S$, of genus $12n\beta+8n\alpha-8n+1$, such that ${\rm Aut}(S) =K_{n}$ and $G_n<{\rm Aut}^{+}(S)=H_{n}$.
 \end{theo}
 \begin{proof}
 Let $n\geq 2$ be an integer and let us consider an NEC group $\Delta$ with signature $(1;-;[2,\stackrel{\alpha}{2,\cdots, 2},\stackrel{\beta}{4,\cdots, 4}];\{-\})$ and presentation 
 $\Delta=\langle d_1,a_0,a_1,\cdots, a_{\alpha},b_1,\cdots, b_{\beta}\rangle,$ where its generators satisfy the following relations
 \begin{center}$a_0^2=a_1^2=\cdots=a_{\alpha}^2=b_1^4=\cdots=b_{\beta}^4=1,a_0\prod_{i=1}^{\alpha}a_i\prod_{j=1}^{\beta}b_jd_1^2=1,$\end{center}
 where $\alpha\geq 1$ be an integer and $\beta\geq 2$ be an even integer. Then quotient Klein surface uniformized by $\Delta$ is the orbifold whose underling surface is the real projective plane its conical points are $1+\alpha+\beta$. Let us consider the epimorphism $\theta:\Delta\to K_{n}$ given by
 
 \noindent (if $\alpha$ odd and $m$ be an integer)\\
 $\theta(d_1)=zx^m$; $\theta(a_0)=z^2yx^2$; $\theta(a_1)=y$; $\theta(a_i)=yx^{m_i}$; $\theta(a_{i+1})=yx^{m_i}$, $m_i$-even, $2\leq i\leq \alpha-1$; $\theta(b_1)=yx$; $\theta(b_2)=yx^{2n-2m+3}$; $\theta(b_j)=yx^{m_{j}}$; $\theta(b_{j+1})=(yx^{m_j})^{-1}$, $m_j$-odd, $3\leq j\leq \beta-1$.
 
 \medskip
 \noindent (if $\alpha$ even and $m$ be an integer)\\
  $\theta(d_1)=zx^m$; $\theta(a_0)=z^2yx^2$; $\theta(a_1)=y$; $\theta(a_2)=x^{2n}$; $\theta(a_i)=yx^{m_i}$; $\theta(a_{i+1})=yx^{m_i}$, $m_i$-even, $3\leq i\leq \alpha-1$; $\theta(b_1)=yx$; $\theta(b_2)=yx^{-2m+3}$; $ \theta(b_j)=yx^{m_{j}}$; $\theta(b_{j+1})=(yx^{m_j})^{-1}$, $m_j$-odd, $3\leq j\leq \beta-1$.

The kernel $\Gamma$ of $\theta$ is a torsion-free subgroup (contained in the half-orientation part $\Delta^{+}$ of $\Delta$) such that $S=\mathbb{H}/\Gamma$ is a closed Riemann surface with $K_{n}\leq {\rm Aut}(S)$ and $G_n< H_{n} \leq {\rm Aut}^{+}(S)$. As the signature of $S/H_{n}$  is $(0; 2,2,\stackrel{2\alpha}{2, \cdots, 2}, \stackrel{2\beta}{4, \cdots,4})$, where the number of cone points is exactly $2+2\alpha+2\beta\geq 6$, it follows (from Singerman list of maximal Fuchsian groups \cite{S.}) that we may choose $\Delta$ so that ${\rm Aut}(S)=K_{n}$. In this case, as the only anticonformal  automorphisms of $S$ are the elements of $K_{n}\setminus H_{n}$ (which have order different than two) it follows that $S$ is a pseudo-real Riemann surface. The Riemann-Hurwitz formula, applied to the branched regular cover map $S \rightarrow S/H_{n},$ permits to obtain that $S$ has genus $12n\beta+8n\alpha-8n+1$.

\end{proof}

 \subsubsection{\bf Case $n \geq 2$ even}
 Let us recall that a $2$-group $P$ is {\em exceptional} if it is either cyclic, generalized quaternion, dihedral, or quasi-dihedral group.

 \begin{lemma}[Lemma 2.2, \cite{Mu}]\label{leg1}
 Let $G$ be a finite group with a Sylow $2$-subgroup $P$ of order $2^m$ ($m\geq 2$), and let ${\rm inv}(G)$ be the number of involutions in $G$. 
  \begin{enumerate}
  \item If $P$ is non-exceptional, then ${\rm inv}(G)\equiv 3 \ {\rm mod} (4)$.
  \item if $P$ is exceptional, then ${\rm inv}(G)\equiv 1\ {\rm mod} (4)$.
  \end{enumerate}
  \end{lemma}
  
In Theorem \ref{ejemplo}, we constructed pseudo-real Riemann surfaces $S$ admitting $G_{n}$ as index two subgroup of ${\rm Aut}^{+}(S)$. In the next, we observe that is the best we can do, that is, there is no pseudo-real surface with $G_{n}$ as its full group of conformal automorphisms.

\begin{theo}\label{tps2}
If $n \geq 2$ even, then there is not a pseudo-real Riemann surface $S$ with $G_{n} \leq {\rm Aut}^{+}(S)$ and with some $\alpha \in {\rm Aut}(S) \setminus {\rm Aut}^{+}(S)$ such that $\alpha^{2} \in G_{n}$. In particular, there is no pseudo-real Riemann surface with
$G_{n} = {\rm Aut}^{+}(S)$.
\end{theo}

\begin{proof}
Let $n=2^{r}k$ with ${\rm gcd}(2,k)=1$ and $r\geq 1$.
Suppose, by the contrary, that there exists a pseudo-real Riemann surface $S$ with $G_{n} \leq {\rm Aut}^{+}(S)$  and with some $\alpha \in {\rm Aut}(S) \setminus {\rm Aut}^{+}(S)$ such that $\alpha^{2} \in G_{n}$.
If $G=\langle \alpha, G_{n} \rangle$, then $[G:G_{n}]=2$. As $G_{n}$ is a non-abelian group of order $8n$, then $G$ is a non-abelian group of order $16n=2^{r+4}k$. Moreover, $G$ does not admit an exceptional $2$-Sylow subgroup \cite{ABG, GW}. So, by Lemma \ref{leg1}, we have  ${\rm inv(G)}\equiv\ 3\ {\rm mod} (4)$.
Since $G_{n}$ has $2^{r+1}k+1$ involutions (Lemma \ref{pgg1}) and, as $2^{r+1}k+1\not\equiv 3\ {\rm mod} (4)$, it follows from Lemma \ref{ABQ}  that $G_{n}$ cannot act as group of conformal automorphisms of a pseudo-real Riemann surface. 
\end{proof}

\subsubsection{\bf Case $n \geq 3$ odd}
For each odd integer $n \geq 3$, set
$\widehat{G}_{n}=\langle z, y:\ z^{8n}=y^2=1, yzy=z^{2n-1}\rangle$. Note that it contains as index two subgroup $G_{n}=\langle x=z^2, y\rangle$.

\begin{theo}\label{tps1}
Let $l\geq 2$ be an integer, $n\geq 3$ and $r\geq 1$ odd integers. There are closed pseudo-real Riemann surfaces $S$, of genus $4nl+6nr-8n+1$, such that ${\rm Aut}(S) =\widehat{G}_{n}$ and ${\rm Aut}^{+}(S) =G_{n}$.
\end{theo}

\begin{proof}
Let $n\geq 3$ be an odd integer and let us consider an NEC group with signature $(1;-;[\stackrel{l}{2,\cdots,2},\stackrel{r}{4,\cdots,4}];\{-\})$, whose presentation is 
$\Delta=\langle d_1, a_1, \cdots, a_l, b_1,\cdots,b_r :\ a_1^2=\cdots=a_l^2=b_1^4=\cdots=b_r^4=1,\prod_{j=1}^la_j\prod_{i=1}^rb_id_1^2=1\rangle,$ where $l\geq 2$ be an integer and $r\geq 1$ be an integer odd. 
The quotient Klein surface $\h/\Delta$ is an orbifold whose underlying surface is the real projective plane and its conical points are $l+r$ points. 

Let us consider the epimorphism $\theta:\Delta\to \widehat{G}_{n}$ given by

 \noindent{(if $l$ even)}\begin{small}
 $$\theta(d_1)=z, \, \theta(a_1)=yz^{2n+2}, \theta(a_j)=y\ (2\leq j\leq l), 
 \theta(b_1)=z^{2n},\theta(b_{s})=z^{2n},\theta({b_{s+1}})=z^{-2n}\ (2\leq s\leq r-1);$$ 
 \end{small}
 \noindent{(if $l$ odd)} \begin{small}
 $$\theta(d_1)=yz, \, \theta(a_l)=z^{4n}, \theta(a_j)=y\ (1\leq j\leq l-1), 
 \theta(b_1)=z^{2n},\theta(b_{s})=z^{2n},\theta({b_{s+1}})=z^{-2n}\ (2\leq s\leq r-1).$$
  \end{small}
The kernel $\Gamma$ of $\theta$ is a torsion-free subgroup (contained in the half-orientation part $\Delta^{+}$ of $\Delta$) such that $S=\mathbb{H}/\Gamma$ is a closed Riemann surface with $\widehat{G}_{n} \leq {\rm Aut}(S)$ and $G_{n} \leq {\rm Aut}^{+}(S)$. As the signature of $S/\langle z^2, y\rangle$  is $(0; \stackrel{2l}{2, \cdots, 2}, \stackrel{2r}{4, \cdots,4})$, where the number of cone points is exactly $2l+2r\geq 6$, it follows (from Singerman list of maximal Fuchsian groups \cite{S.}) that we may choose $\Delta$ so that ${\rm Aut}(S)=\widehat{G}_{n}$. In this case, as the only anticonformal  automorphisms of $S$ are the elements of $\widehat{G}_{n} \setminus G_{n}$ (which have order different than two) it follows that $S$ is a pseudo-real Riemann surface. The Riemann-Hurwitz formula, applied to the branched regular cover map $S \rightarrow S/ \langle z^2, y \rangle,$ permits to obtain that $S$ has genus $4nl+6nr-8n+1.$
\end{proof}

\begin{theo}\label{mps+} Let $n \geq 3$ be an odd integer.  The minimal genus of a closed pseudo-real Riemann surface $S$ such that $\widehat{G}_{n}\leq {\rm Aut}(S)$ and $G_{n}=\widehat{G}_{n}\cap {\rm Aut}^{+}(S)$  is $6n+1$.
In that case, $S/G_{n}=\h/\Delta^+$ has signature $(0;+;[2,2,2,2,4,4];\{-\})$ and $S/\widehat{G}_{n}=\h/\Delta$ has signature $(1;-;[2,2,4];\{-\})$. 
\end{theo}
\begin{proof} This follows by the Riemann-Hurwitz formula and hyperbolic area comparison. The quotient orbifold $S/G_n=\h/\Delta^+$ admits an imaginary reflection $\rho$, induced by $z$. It follows that its signature is
$(\gamma;+;[\stackrel{2\alpha}{2,\cdots, 2}, \stackrel{2\beta}{4,\cdots, 4},\delta_1,\delta_1,\delta_2,\delta_2,\cdots, \delta_l,\delta_l];\{-\}),$ where each cone point of order 2  is produced by a fixed point of one of the involutions $yz^{4t}$, with $0\leq t\leq 2n-1$; each cone point of order 4 is produced by a fixed point of $yz^{4l+2}$, with $0\leq l\leq 2n-1$; and each point of order $\delta_j$ is produced by a fixed point of some element in $\langle z^2\rangle$ (i.e., $\delta_j\geq 2, \delta_j| 4n$). By the Riemann-Hurwitz formula we have
\begin{equation}\label{rhf}
2(g-1) = 16n(\gamma-1)+8n\alpha+12n\beta+2C,
\end{equation} where $C=\sum_{j=1}^l\frac{8n}{\delta_j}(\delta_j-1)$.
We want to minimize the reduced area $|\Delta^+|_{\gamma}^*$ (the right side of equation (\ref{rhf})) of $\Delta^+$, where
\begin{equation}\label{rhf1}
|\Delta^+|_{\gamma}^*=16n(\gamma-1)+8n\alpha+12n\beta+2C.
\end{equation}

In this case, the quotient orbifold $S/\widehat{G}_{n}=\h/\Delta$ has signature 
$$(\gamma+1;-;[\stackrel{\alpha}{2,\cdots, 2}, \stackrel{\beta}{4,\cdots, 4},\delta_1,\delta_2,\cdots, \delta_l];\{-\}).$$ 
 So, $\Delta$ has a presentation 
$$\Delta=\langle d_1,\cdots, d_{\gamma+1}, a_1,\cdots, a_{\alpha},b_1,\cdots, b_{\beta}, c_1,\cdots, c_{l}\rangle,$$ and its generators satisfy the following relations
$$a_1^2=\cdots =a_{\alpha}^2=b_1^4=\cdots=b_{\beta}^4=c_1^{\delta_1}=\cdots=c_{l}^{\delta_l}=1,\prod_{i=1}^{\alpha}a_i\prod_{j=1}^{\beta}b_j\prod_{k=1}^{l}c_k\prod_{s=1}^{\gamma+1}d_s^2=1,$$ where $S=\h/{\rm ker}(\theta)$  for a surjective homomorphism $\theta:\Delta\to \widehat{G}_n$ such that $\theta(\Delta^+)=G_n$ and
$$\theta(d_j)\in\{y^{t_j}z^{2s_j-1},\ s_j\in\{1,\cdots, 4n\}, t_j\in\{0,1\}\};\ \theta(a_j)\in \{yz^{4r_j}, r_j\in\{0,\cdots, 2n-1\}\};$$ $$\theta(b_j)\in\{yz^{4s_j+2}, s_j\in\{0,\cdots, 2n-1\}\},\ \theta(c_j)\in \langle z^2\rangle.$$

The quotient orbifold $S/G_n=\h/\Delta^+$ admite the signature $(0;+;[2,2,2,2,4,4];\{-\})$ as consequence of the Theorem \ref{tps1}.  In this case, the reduced area of $\Delta^+$ is 
\begin{equation}
|\Delta^+|_0^*=16n(0-1)+16n+12n=12n,
\end{equation}
from which $S$ has genus equal to $6n+1$. Moreover, the signature $(1;-;[2,2,4];\{-\})$ is admisible for the quotient orbifold $S/\widehat{G}_{n}$.
We proceed to check that this signature has the least reduced area.

\noindent{\bf Case $\gamma\geq 2$}. In this case, the reduced area of $\Delta^+$ is $|\Delta^+|_{\gamma\geq 2}^*=16n+8n\alpha+12n\beta+2C\geq 16n$, so $|\Delta^+|_{\gamma\geq 2}^*>|\Delta^+|_0^*$.

\noindent{\bf Case $\gamma=1$}. Suppose that $|\Delta^+|_{1}^*=8n\alpha+12n\beta+2C\leq 12n$. Then we have the cases
$$i)\ \beta=0, \alpha=1, C\leq 2n;\quad ii)\ \beta=1,\alpha=0, C=0; \quad iii)\ \beta=0,\alpha=0, C\leq 6n.$$

In case $i)$, we may observe that $C\leq 2n$ implies that $l=0$. So $\Delta$ has signature $(2;-;[2];\{-\})$ and an epimorphism $\theta:\Delta\to \widehat{G}_n$ is given by $\theta(d_j)=y^{t_j}z^{2s_j-1}$; $\theta(a_1)=yz^{4r}$. Then the relation $\theta(a_1)\theta(d_2)^2\theta(d_1)^2=1$ implies that
$$1=(yz^{4r})(y^{t_1}z^{2s_1-1})^2(y^{t_2}z^{2s_2-1})^2=yz^m\quad \quad (\text{a contradiction}).$$ 
In case $ii)$, $\Delta$ has signature $(2;-;[4];\{-\})$ and an epimorphism $\theta:\Delta\to \widehat{G}_n$ is given by $\theta(d_j)=y^{t_j}z^{2s_j-1}$; $\theta(b_1)=yz^{4r+2}$. Then the relation $\theta(b_1)\theta(d_2)^2\theta(d_1)^2=1$ implies that
$$1=(yz^{4r+2})(y^{t_1}z^{2s_1-1})^2(y^{t_2}z^{2s_2-1})^2=yz^m\quad (\text{a contradiction}).$$ 
 In case $iii)$, $\Delta$ has signature $(2;-;[\delta_1,\cdots, \delta_l];\{-\})$ and an epimorphism $\theta:\Delta\to \widehat{G}_n$ is given by $\theta(d_j)=y^{t_j}z^{2s_j-1}$; $\theta(c_j)\in \langle z^2\rangle$, where we get $y\not\in \theta(\Delta)$ (a contradiction).

\noindent{\bf Case $\gamma=0$}. Suppose that $|\Delta^+|_0^*=-16n+8n\alpha+12n\beta+2C\leq12n$. Then we have 
$$8n\alpha+12n\beta+2C\leq 28n,$$ so $\alpha\leq 3$.
(1) If $\alpha=3$, then $\beta=0$, $C=0$ ($C\leq 2n$ implies that $l=0$), and $\Delta$ has signature $(1;-;[2,2,2];\{-\})$ and an epimorphism  $\theta:\Delta\to \widehat{G}_n$ is given by $\theta(a_j)=yz^{4r_j}$; $\theta(d_1)=y^{t_1}z^{2s_1-1}$. Then the relation $\theta(a_1)\theta(a_2)\theta(a_3)\theta(d_1)^2=1$ implies that
$$1=(yz^{4r_1})(yz^{4r_2})(yz^{4r_3})(y^{t_1}z^{2s_1-1})^2=yz^m\quad (\text{a contradiction}).$$

(2) If $\alpha=2$, then $\beta=1$, $C=0$, and $\Delta$ has signature $(1;-;[2,2,4];\{-\})$ (desired signature). 

(3) If  $\alpha=1$, then $12n\beta+2C\leq 20n$. So, we have the cases $$i)\ \beta=0, C\leq 10n;\quad ii)\ \beta=1, C\leq 4n.$$

In case $i)$,  $\Delta$ has signature $(1;-;[2,\delta_1,\cdots, \delta_l];\{-\})$ and an epimorphism $\theta:\Delta\to \widehat{G}_n$ is given by $\theta(d_1)=y^{t_1}z^{2s_1-1}$; $\theta(a_1)=yz^{4r_1}$; $\theta(c_j)\in \langle z^2\rangle$. 
The relation $\theta(a_1)\prod_{i=1}^l\theta(c_i)\theta(d_1)^2=1$ implies that
$1=yz^{4r_1}z^{2A}(y^{t_1}z^{2s_1-1})^2=yz^m\quad (\text{a contradiction}).$
In case $ii)$,  $\Delta$ has signature $(1;-;[2,4,\delta_1,\cdots, \delta_l];\{-\})$ and an epimorphism $\theta:\Delta\to \widehat{G}_n$ is given by $\theta(d_1)=y^{t_1}z^{2s_1-1}$; $\theta(a_1)=yz^{4r_1}$; $\theta(b_1)=yz^{4s_2+2}$ and $\theta(c_j)\in \langle z^2\rangle$. As $C\leq 4n$, this implies that $l=0$ or $1$. 
If $l=0$, then (by results in \cite{Bu0}), the action of the group $\widehat{G}_n$ with signature $(1;-;[2,4];\{-\})$ on $S$ can be extends to the action of a larger group on $S$ with signature $(0;+;[2];\{(2,4)\})$, and again $S$ admits reflections, and so cannot be pseudo-real. For $l=1$, the signature $(1;-;[2,2,4];\{-\})$ is the desired signature. 

(4) If $\alpha=0$, then $12n\beta+2C\leq 28n$, so
$$i)\ \beta=2,\ C\leq 2n;\quad ii)\ \beta=1, C\leq 8n;\quad iii)\ \beta=0, C\leq 14n.$$
In this case, $\Delta$ has signature $(1;-;[\stackrel{\beta}{4,\cdots, 4},\delta_1,\cdots, \delta_l];\{-\})$ and an epimorphism $\theta:\Delta\to \widehat{G}_n$ is given by $\theta(d_1)=y^{t_1}z^{2s_1-1}$; $\theta(b_j)=yz^{4t+2}$ and $\theta(c_j)\in \langle z^2\rangle$.
In case $i)$, $\Delta$ has signature $(1;-;[4,4];\{-\})$ ($C\leq 2n$ implies that $l=0$). By results in \cite{Bu0}, the action of the group $\widehat{G}_n$ with signature $(1;-;[4,4];\{-\})$ on $S$ can be extends to the action of a larger group on $S$ with signature $(0;+;[2,4];\{(-)\})$, and again $S$ admits reflections, and so cannot be pseudo-real.
In case $ii)$, the relation  $\theta(b_1)\prod_{j=1}^l\theta(c_j)\theta(d_1)^2=1$ implies that $1=yz^m$\  (contradiction).
In case $iii)$, the relation $\prod_{j=1}^l\theta(c_j)\theta(d_1)^2=1$ implies that $y\not\in \theta(\Delta)$\  (contradiction).
\end{proof}

\begin{remark}[Non-uniqueness of the action]\label{Ups+}
There are different topological actions of $\widehat{G}_n$ on pseudo-real Riemann surfaces of genus $6n+1$ with $G_{n}$ as its conformal part.
In fact, by Theorem \ref{mps+}, we can consider an NEC group $\Delta$ with signature $(1;-;[2, 2, 4];\{-\})$ and presentation $\Delta=\langle d_1,a_1,a_2,b_1:\ a_1^2=a_2^2=b_1^4=1,a_1a_2b_1d_1^2=1 \rangle.$
We want to describe the epimorphisms $\theta:\Delta\to \widehat{G}_n$ such that $\theta(\Delta^{+})=G_{n}$ and torsion-free kernel. It must satisfy that 
$$\theta(a_i)\in \{x^{4n},yx^{4r_i},\ r_i\in\{0,\cdots, 2n-1\}\};\ \theta(b_1)\in \{x^{2n}, x^{-2n},yx^{4s_1+2},\ s_1\in\{0,\cdots, 2n-1\}\};$$ 
$$\theta(d_1)\in\{y^{t_1}z^{2s_1-1},\ s_1\in\{1,\cdots, 4n\}, t_1\in\{0,1\}\};\ \theta(a_1)\theta(a_2)\theta(b_1)\theta(d_1)^2=1.$$

By post-composing $\theta$ by automorphisms of $G_{n}$, we obtain that all of these possibilities for $\theta$ are ${\rm Aut}(G_{n})/{\rm Aut}(\Delta)$-equivalent to:
$$\theta_1:\Delta\to G_{n}:\ \theta_1(d)=x;\ \theta_1(b_1)=x^{2n};\ \theta_1(a_1)=yx^{2n+2};\ \theta(a_2)=y,$$ or
$$\theta_2:\Delta\to G_{n}:\ \theta_2(d_1)=yx;\ \theta_2(b_1)=yx^{2n};\, \theta_{2}(a_1)=y;\ \theta_2(a_2)=x^{4n}.$$
\end{remark}

\section{Minimal genus actions of $G_{n}$ on compact Klein surfaces} 
Let $S$ be a closed Riemann surface of genus $g\geq 2$ and $\tau$ be an anticonformal  involution of $S$. The quotient orbifold $X=S/\langle \tau \rangle$ is a compact hyperbolic Klein surface whose underlying topological surface is either (i) non-orientable with empty boundary ({\em closed Klein surface}) or (ii) it has non-empty boundary (it might be orientable or not); the fixed points of $\tau$ correspond to the boundary points of $X$. The surface $S$ is called an ``analytic double$"$ of $X$ and $g$ its  {\em algebraic genus}. The {\em topological genus} of  $X$ is  $\gamma=(g-k+1)/\eta$, where $\eta= 2$ if $S/\langle \tau \rangle$ is orientable and $\eta=1$ otherwise, and 
 $k$ is the number of connected components of the fixed points of $\tau$. Moreover, it is known that  ${\rm Aut}(X) \cong \{f\in {\rm Aut}^{+}(S) :\ f\tau=\tau f\}.$ 
 
 By the uniformization theorem, there exists an NEC group $\Gamma$ such that $X=\h/\Gamma$. 
 Let $G$ be a given finite group. It is well known that $G$
can be realized as a subgroup of ${\rm Aut}(X)$ if and only if there exists an NEC group $\Delta$, containing $\Gamma$ as a normal subgroup, and there exists a surjective homomorphism $\theta:\Delta \to G$ with $\Gamma$ as its kernel (see, for instance, \cite{BuEtaGamGro,Pr}). 
From the relation between hyperbolic areas we have $g-1=|G||\Delta|^*,$ where $|G|$ is the order of $G$ and $|\Delta|^{*}$ is the reduced area of $\Delta$ (see, Section \ref{S2}). To minimize the value of $g$, keeping fixed $G$, we need to minimize the value of $|\Delta|^{*}$.

 In \cite{May77}, it was observed that $G$ can be realized as a group of automorphisms of a hyperbolic bordered Klein surface. The minimum algebraic genus $\rho(G)$ of these surfaces is called the {\em real genus} of $G$. Similarly, in \cite[Theorem 2.5]{Bu}, it was observed that 
$G$ can also be realizable as a group of automorphisms of some closed Klein surface. The minimal topological genus $\tilde{\sigma}(G)$ of these surfaces is called the {\em symmetric crosscap number} of $G$. It is know that $3$ cannot be the symmetric crosscap number of any group \cite{May0}. In the following, we proceed to compute these values for $G=G_{n}$, with $n \geq 2$.

\begin{theo}\label{real}
If $n \geq 2$, then 
\begin{enumerate}
\item[(a)] $\rho(G_{n})=2n+1$ and the action of $G_{n}$ is unique.
\item[(b)] $\tilde{\sigma}(G_{n})=2n+2$ and the action of $G_{n}$ is not unique (for $n \neq 3$).
\end{enumerate}
\end{theo}

\begin{remark}
Part (a), when $n$ is a power of two (respectively, for $n \geq 3$ odd), 
was previously obtained in \cite{May94} (respectively, \cite{EM}). 
\end{remark}

\subsection{Proof of part (a) of Theorem \ref{real}}
In order to compute 
$\rho(G_n)$, we must find the minimal possible value of $|\Delta|^*$ among those NEC groups  $\Delta$ admitting an epimorphism to $G_n$ and whose kernel has 
signature of the form $(\gamma;\pm; [-],\{\stackrel{k}{(-)\cdots (-)}\}).$
We recall that the signature of $\Delta$ must contain an empty period-cycle or a period cycle with two consecutive link-periods equal to 2 (see, \cite{BM89}). 
In \cite{May93}, it was noted that 
\begin{center}$\rho(C_{4n})=\rho(D_{2n})=0,\ \rho(DC_{4n})=\begin{cases}
2n+1,& \text{if $n\not=3$}\\
 6,& \text{if $n=3$}
 \end{cases},$\end{center}      
 and $\rho(H)\leq \rho(G_{n})$ for every subgroup $H$ of $G_{n}$. 
It follows, from the above, that 
 $6\leq \rho(G_3)$ and $2n+1\leq \rho(G_{n})$, for $n\not=3$. As in the group $G_{n}$, the element $xy$ has order $4$ (Lemma \ref{pgg1}) and this group can be generated by $y$ and $xy$, it is sufficient to consider an NEC group $\Delta$ with signature $(0;+;[2,4],\{(-)\})$ and presentation
\begin{equation}\label{rmg}\Delta=\langle \beta_1, \beta_2, e_1, c_{10}, c_{11}:\ \beta_1^2=\beta_2^4=c_{1j}^2=c_{10}c_{11}=1, e_1c_{10}e_1^{-1}c_{11}=\beta_1\beta_2e_1=1\rangle.\end{equation}

Let us consider the surjective homomorphism $\theta:\Delta\to G_{n}$, given by
$$\theta(\beta_1)=y;\ \theta(\beta_2)=xy;\ \theta(e_1)=x^{2n+1};\ \theta(c_{10})=1;\ \theta(c_{11})=1.$$ The group $\Delta$ has reduced area $1/4$ (see, \cite{ECM}), and this is the lowest possible. Thereby, $\rho(G_{n})\leq 2n+1$ for all $n\geq 2$. So $\rho(G_{n})= 2n+1$.
Also, note that the epimorphism $\theta$ is valid for $n=3$, so $\rho(G_{3})=7$.

In order to prove the uniqueness of the $G_{n}$-action we proceed as follows.
As noted above, the minimal action is provided by an NEC group $\Delta$ with signature $(0;+;[2,4];\{(-)\})$ and presentation as in (\ref{rmg}). Let us consider a surjective homomorphism $\theta:\Delta\to G_{n}$.  It must satisfy that $\theta(\beta_{1})\! \in\! \lbrace x^{2n}, yx^{2s}, s\!\in\! \lbrace0, 1, \tiny{\ldots} , 2n-1 \rbrace \rbrace;$ $\theta(\beta_{2})\! \in\! \lbrace x^{n}, x^{-n}, yx^{2t+1},\ t\in \lbrace0, 1, \tiny{\ldots} , 2n-1 \rbrace \rbrace;$ $\theta(e_1)\!\in\! \Delta^{+}$; $\theta(c_{1j})\!\in\! \Delta\setminus \Delta^{+};$ $\theta(\beta_{1})\theta(\beta_{2}) \theta(e_{1})=\theta(e_{1})\theta(c_{10}) \theta(e_{1})^{-1}\theta(c_{11})=1.$

By post-composing $\theta$ by automorphisms of $G_{n}$, we may obtain either 
\begin{center}$\theta_1:\Delta\to G_{n}:\ \theta_1(\beta_1)=y;\ \theta_1(\beta_2)=xy;\ \theta_1(e_1)=x^{2n+1};\ \theta_1(c_{10})=\theta_1(c_{11})=1,\ \text{or}$\end{center}
\begin{center}$\theta_2:\Delta\to G_{n}:\ \theta_2(\beta_1)=x^{2n};\ \theta_2(\beta_2)=yx;\ \theta_2(e_1)=yx;\ \theta_2(c_{10})=y,\ \theta_2(c_{11})=yx^{2n+2}.$\end{center}

Moreover, if $L$ is the automorphism of $\Delta$ given by $$L(\beta_{1})=\beta_2^2; \, L(\beta_2)=\beta_2^{-1}; \, L(e_1)=\beta_2^{-1}; \, L(c_{10})=\beta_1c_{10};\ L(c_{11})=\beta_2^{-1}c_{10}\beta_1\beta_2,$$ 
then $\theta_1\circ L:\Delta\to G_{n}$ is given by 
\begin{small}
$$ (\theta_1\circ L)(\beta_1)=x^{2n};\ (\theta_1\circ L)(\beta_2)=x^{2n+1}y;\, (\theta_1\circ L)(e_1)=yx^{-1};\ (\theta_1\circ L)(c_{10})=y;\ (\theta_1\circ L)(c_{11})=yx^{2n-2}.$$
\end{small}
We note that, after post-composition of $\theta_1\circ L$ by the automorphism $\psi_{4n-1,0}$ of $G_{n}$ we may obtain $\theta_2$. 

\subsection{Proof of part (b) of Theorem \ref{real}}
If  $X = \h/\Gamma$ is a closed Klein surface, of topological genus $\gamma \geq 3$, such that $G \leq {\rm Aut}(X)$, then there exists an NEC group $\Delta$ such that $\Gamma$ is a normal subgroup of $\Delta$ with $G = \Delta / \Gamma$. Moreover, from the Riemann-Hurwitz formula we have $\gamma-2 = |G| |\Delta|^*$. In this way, to compute $\tilde{\sigma}(G)$, we need to minimize the value of $|\Delta|^*$ among all those possible NEC groups $\Delta$. 

It is well known that if $H$ is a subgroup of $G$, then 
\begin{equation}\label{mcg}
\tilde{\sigma}(H)\leq \tilde{\sigma}(G).
\end{equation}
     
If $\tilde{\sigma}(H)>1$, for $H \in\lbrace C_{4n}, D_{2n}, DC_{4n} \rbrace$, then (see, \cite{Bu, ECM}) we have that $\tilde{\sigma}(C_{4n})=2n+1$,
 $$
 \tilde{\sigma}(D_{2n})=\left\{\begin{array}{ll}
 n+1,& \mbox{if $4$ divides $2n$}\\
  n,& \mbox{otherwise}
 \end{array}\right\}, \;
 \tilde{\sigma}(DC_{4n})=\left\{\begin{array}{ll}
 2n+2,& \mbox{if $n\not=3$}\\
7,& \mbox{if $n=3$}
\end{array}
\right\}.$$

The above, together with the inequality (\ref{mcg}), asserts that $2n+2 \leq \tilde{\sigma}(G_{n})$ for $n\not=3$, and  $7 \leq \tilde{\sigma}(G_3).$ In this way, we only need to centre our attention to the actions of the subgroup $DC_{4n}$ on closed Klein surfaces. Below, we observe that $\tilde{\sigma}(G_{n})\leq 2n+2$, from which the desired equality follows.

\subsection*{Case $n \neq 3$}
An action of $DC_{4n}$ can be carried out by an NEC group with signature $(0;+;[4,4]; \lbrace (-)\rbrace)$. By results in \cite{BCC}, this signature can be extended with index two to the signatures:
 $$i)\ (0;+; [2,4];\lbrace (-)\rbrace),\quad ii)\ (0;+; [4];\lbrace (2,2)\rbrace).$$ 

\medskip
\noindent Case $i)$.  We have an NEC group $\Delta$ with presentation
$$\Delta=\langle \beta_1, \beta_2, c_{10}, c_{11}, e_1:\ \beta_1^2=\beta_2^4=c_{1j}^2=c_{10}c_{11}=1, \, e_1c_{10}e_1^{-1}c_{11}=\beta_1\beta_2e_1=1\rangle,$$ 
and an epimorphism $\theta:\Delta \to G_{n}$ given by $$\theta(\beta_1)=y;\ \theta(\beta_2)=yx;\ \theta(e_1)=x^{-1};\ \theta (c_{10})=\theta (c_{11})=x^{2n}.$$ 
It is clear that the image of the elements $\beta_{1}$ and $\beta_1\beta_2$ generate the group $G_{n}$ and both preserve the orientation, so the group $G_{n}$ acts on non-orientable surfaces. By the minimality of $\tilde{\sigma}(G_{n})$, and results on reduced areas in \cite{ECM}, we have $\tilde{\sigma}(G_{n})\leq 2+ \mid G_{n} \mid 1/4 =2+2n.$ 

\medskip
 \noindent Case $ii)$.  We have an NEC group $\Delta$ with presentation
 $$\Delta=\langle \beta_1, c_{10}, c_{11},c_{12}, e_1:\ \beta_1^4=c_{1j}^2=(c_{10}c_{11})^2=(c_{11}c_{12})^2=1, e_1c_{10}e_1^{-1}c_{12}=\beta_1e_1=1\rangle,$$ 
 and an epimorphism $\theta:\Delta \to G_{n}$ given by $$\theta(\beta_1)=xy; \,\theta(e_1)=yx^{4n-1}; \, \theta(c_{10})=y; \, \theta(c_{11})=x^{2n};\theta(c_{12})=yx^{2n-2}.$$ 
 
The image of the elements $c_{10}$ and $x_1c_{10}$ generate the group $G_{n}$.  Now we have that the element $(x_1c_{10})^{2n}c_{11}$ has as image the identity, and is an element that reverse orientation, so the group $G_{n}$ acts on non-orientable surfaces. Again, by the minimality of $\tilde{\sigma}(G_{n})$, and results on reduced areas in \cite{ECM},  we have $\tilde{\sigma}(G_{n})\leq 2+ \mid G_{n} \mid 1/4 =2+2n.$

\begin{remark} 
As a consequence of the above, we obtain that on the symmetric crosscap number of $G_{n}$ (for $n \neq 3$) there are two different signatures.
\end{remark}

\subsection*{Case $n=3$}
In this case, the action of the group $DC_{12}$ on closed Klein surfaces, can be realized with an NEC group with signature $(0;+;[3,4]; \lbrace (-)\rbrace)$. By results in \cite{BCC}, this can be extended with index two to an NEC group with signature $(0;+; [-];\lbrace (2,2,3,4)\rbrace)$. Since $G_3$ cannot be generated by involutions, we cannot build an epimorphism $\theta:\Delta \to G_3$. In \cite{Ba}, there is a classification of the groups with symmetric crosscap number equal $7$.  The group $G_3$ does not belong to this classification.
Note that the epimorphisms, previously constructed for the case $n \neq 3$, is still valid for $n=3$. So, $\tilde{\sigma}(G_3)=8$.
    
\medskip
\noindent{\bf Acknowledgements}\\
The results of this article are mostly based on the second author's Ph.D. thesis. 


\begin{thebibliography}{99}
\bibitem{ABCNPW}
M. Arbo, K. Benkowski, B. Coate, H. Nordstrom, C. Peterson and A. Wootton.
The genus level of a group. 
{\it Involve a journal of mathematics} {\bf 2} (2009), 323--340.

\bibitem{AQR}
M. Artebani, S. Quispe and C. Reyes.
Automorphism groups of pseudoreal Riemann surfaces.
{\it Journal of Pure and Applied Algebra} {\bf 221} (2017), 2383--2407.

\bibitem{ABG}
J. L. Alperin, R. Brauer and D. Gorenstein.
{\it Finite groups with quasi-dihedral and wreathed Sylow 2-subgroups}.
Transactions of the American Mathematical Society, {\bf 151} (1970).


\bibitem{Ba}
A. Bacelo.
The full group of automorphisms of non-orientable unbordered Klein surfaces of topological genus 7.
{\it  Rev Mat Complut.} {\bf 31} (2018), 247--261.

\bibitem{BGH}
C. Bagi\'nski, G. Gromadzki and R. A. Hidalgo.
On purely non-free finite actions of abelian groups on compact surfaces.
{\it Arch. Math.} {\bf 109} (2017), 311--321.


\bibitem{BRT}
A. Behn, A. M. Rojas and M. Tello-Carrera.
A SAGE Package for $n$-Gonal Eqquisymmetric Stratification of $\mathcal{M}_g$.
{\it  Experimental Mathematics} (2020), 1--16. 

\bibitem{Bu03}
E. Bujalance, F. J. Cirre, J. M. Gamboa and G. Gromadzki.
On compact Riemann surfaces with dihedral groups of automorphisms.
{\it Math. Proc. Cambridge. Philos. Soc.} {\bf 134} (2003), 465--447. 

\bibitem{Bu10}
E. Bujalance, M. D. E. Conder and A. F. Costa. {Pseudo-real Riemann surfaces and chiral regular maps}, {\it Trans. Am. Math. Soc.}
{\bf 362} (7) (2010), 3365--3376.

\bibitem{BM89}
E. Bujalance and E. Mart\'inez.
A remark on NEC groups representing surfaces with boundary.
{\it Bull. Lond. Math. Soc.} {\bf  21} (1989), 263--266.

\bibitem{Bu0}
E. Bujalance.
Normal N.E.C. signatures.
{\it  Illinois journal of mathematics} {\bf 26}  (1982), 519--530.

\bibitem{Bu}
E. Bujalance.
Cyclic groups of automorphisms of compact nonorientable Klein surface without boundary.
{\it Pacific J. Math.} {\bf 109} (1983),  279--289.
 
 \bibitem{BCC}
 E. Bujalance, F. J. Cirre and M. D. E. Conder.
 Extensions of finite cyclic group action on non-orientable surfaces.
 {\it Pacific J. Amer. Math. Soc.} {\bf  365} (2013), 4209--4227.
 
 \bibitem{BuEtaGamGro}
 E. Bujalance, J. J. Etayo, J. M. Gamboa and G. Gromadzki.
 {\it Automorphism groups of compact bordered Klein surfaces,  A combinatorial approach.}
 Lecture Notes in Mathematics {\bf 1439}, Springer-Verlag, Berlin, 1990.
 
 \bibitem{Burn}
W. Burnside. 
{\it Theory of Groups of Finite Order.}
 Cambridge University Press, 1911.

 
\bibitem{CL}
 M. D. E. Conder and S. Lo.
 The pseudo-real genus of a group. 
 {\it Journal of Algebra} {\bf 561} (2020), 149--162.
 
 
 \bibitem{D}
 R. Diestel.
 {\it Graph Theory}, 3rd ed., Springer, 2005.
 
 \bibitem{EM}
 J.J. Etayo and E. Mart\'inez.
 The real genus of cyclic by dihedral and dihedral by dihedral groups.
 {\it  Journal of Algebra} {\bf 296} (2006), 145--156.
 
 \bibitem{ECM}
 J. J. Etayo, F. J. Cirre and E. Mart\'inez.
 The Symmetric Crosscap Number Of The Groups Of Small-Order.
 {\it Journal of Algebra and Its Applications} {\bf 12} (2013), 125--164 

 
 \bibitem{GiGo}
E. Girondo and G. Gonz\'alez-Diez.
{\it Introduction to compact Riemann surfaces and dessins d'enfants.} 
London Mathematical Society Student Texts  {\bf 79}. Cambridge University Press, Cambridge, (2012).


\bibitem{GW}
D. Gorenstein and J. H. Walter.
The characterization of finite groups with dihedral Sylow 2-subgroups.
{\it  Journal of algebra} {\bf 2} (1965), 85--151.

\bibitem{Gre}
L. Greenberg.
Maximal Fuchsian groups.
{\it Bull. Amer. Math. Soc.} {\bf 69} (1963), 569--573.


\bibitem{Greenberg}
L. Greenberg.
Conformal Transformations of Riemann Surfaces.
{\it Amer. J. of Math.} {\bf 82} (2), 749--760 (1960)


\bibitem{GroT}
J. L. Gross and T. W. Tucker. 
{\it Topological graph theory} (John Wiley and Sons, 1987).

\bibitem{Gro}
A. Grothendieck.
Esquisse d'un Programme (1984). In {\it Geometric Galois Actions}. L. Schneps and P. Lochak eds.  
London Math. Soc. Lect. Notes Ser. {\bf 242}. Cambridge University Press, Cambridge, (1997), 5--47.

\bibitem{Ha}
 W. J. Harvey.
 Cyclic groups of automorphisms of a compact Riemann surface. 
 {\it Quart. J. Math. Oxford Ser.}  {\bf 17} (1966), 86--97.
 
 \bibitem{Hurwitz}
A. Hurwitz. 
\"Uber algebraische gebilde mit eindeutigen transformationen in siche.
{\it Math. Ann.} {\bf 41} (1893), 403--442.


 \bibitem{H.S}
 R. A. Hidalgo and  S. Quispe.
 Regular dessins d'enfants with dicyclic group of automorphisms.  
 {\it Journal of Pure and Applied Algebra} {\bf 224} (2020), 106242.
 
 \bibitem{Hu1}
 B. Huggins.
 Fields of moduli of hyperelliptic curves.
 {\it Math. Res. Lett.} {\bf 14} (2007), 249--262.

 \bibitem{JW}
G. A. Jones and J. Wolfart.
{\it Dessins d'Enfants on Riemann Surfaces.}
Springer Monographs in Mathematics. (2016).

\bibitem{KR}
E. Kani and M. Rosen.
Idempotent relations and factors of Jacobians. 
{\it Math. Ann.} {\bf 284} (1989), 307 -- 327.

\bibitem{LR}
H. Lange and S. Recillas.
Abelian varieties with group actions.
{\it J. Reine Angew. Mathematik} {\bf  575} (2004), 135--155.

  
\bibitem{Mac}
 A. M. Macbeath.
 The classification of non-euclidean plane crystallographic groups.
 {\it Canad. J. Math.} {\bf 19} (1967), 1192--1205. 
   

\bibitem{May77}
C. L. May.
Large automorphism groups of compact Klein surfaces with boundary.
{\it Glasgow Math. J.} {\bf 18} (1977), 1--10.

\bibitem{May93}
C. L. May.
Finite groups acting on bordered surfaces and the real genus of a group.
{\it Rocky Mountain J. Math.} {\bf 23} (1993), 707--724.

\bibitem{MZ7}
C. L. May and J. Zimmerman.
Groups of small strong symmetric genus.
{\it J. Group Theory} {\bf 3} (2000), 233--245.

\bibitem{May94}
C. L. May.
Groups of small real genus.
{\it Houston Journal of Mathematics} no. 3 (1994), 393--408.

\bibitem{May0}
C. L. May.
The symmetric crosscap number of a group.
{\it Glasgow Math. J.} {\bf  43} (2001), 399--410.
   
\bibitem{May3}
C. L. May and J. Zimmerman.
There is a group of every strong symmetric genus.
{\it Bull. London Math. Soc.} {\bf 35} (2003), 433--439.

\bibitem{May10}
C. L. May and J. Zimmerman.
The $2$-groups of odd strong symmetric genus.
{\it Journal of Algebra and Its Applications} {\bf 9} (2010), 465--481.
   

\bibitem{Mu}
M. Murai.
On the number of $p$-subgroups of a finite group.
{\it J. Math. Kyoto Univ.} {\bf 42} (2202), 161--174.


\bibitem{Nag}
S. Nag.
{\it The complex analytic theory of Teichm\"uller spaces.}
A Wiley-Interscience Publication. John Wiley \& Sons, Inc., New York 1988.


\bibitem{Pa}
J. Paulhus.
{\it Elliptic factors in Jacobians of low genus curves}. PhD. thesis, University of Illinois at Urbana-Champaign, 2007.

\bibitem{Pr}
R. Preston.
{\it Projective structures and fundamental domains on compact Klein surfaces}, Ph. D. Thesis. University of Texas (1975).

 \bibitem{AR}
 A. Rojas.
 Group actions on Jacobian varieties.
 {\it Rev. Mat. Iber.} {\bf 23} (2007), 397--420.  
 
 \bibitem{Ro}
 J. C. Rohde.
 {\it Cyclic Coverings, Calabi-Yau Manifolds and Complex Multiplication}. 
 Lecture Notes in mathematics {\bf 1975}, Springer-Verlag, Berlin, 2009.
 

\bibitem{Schwarz}
H. A. Schwartz.
\"Uber diejenigen algebraischen Gleichungen zwischen zwei ver\"anderlichen Gr\"o{\ss}en, welche eine schaar rationaler, eindeutig umkehrbarer 
Transformationen in sich selbst zulassen.
{\it Journal f\"ur die reine und angewandte Mathematik} {\bf 87} (1890), 139--145. 


\bibitem{S.}
D. Singerman. 
Finitely maximal Fuchsian groups.
{\it J. London Math. Soc.} {\bf 6}  (1972), 29--38.


\bibitem{T}
T. W. Tucker. 
Finite groups acting on surfaces and the genus of a group.
{\it J. Combin. Theory Ser. B} {\bf 34} (1983), 82--98.


\bibitem{Wil}
H. C. Wilkie.
On non-Euclidean crystallographic groups.
{\it Math. Z.} {\bf 91} (1996),   87--102.
  
    
\bibitem{Wi}
A. Wiman.
\"Uber die hyperelliptischen Curven und diejenigen von Geschlechte $p=3,$ welche eindeutige Transformationen
in sich zulassen.
{\it Bihang till K. Svenska Vet.-Akad. Handlingar, Stockholm} {\bf 21} (1895), 1--28.
  
  
\end{thebibliography}
\end{document}